\documentclass[11pt,a4paper]{article}

\usepackage{amsmath,amsfonts,amssymb,latexsym,graphics,epsfig,url}
\usepackage{hyperref}
\usepackage{color}
\usepackage{amsthm}
\usepackage{latexsym}
\usepackage{tikz}
\usetikzlibrary{arrows}

\usepackage[english]{babel}
\usepackage{epsfig}
\usepackage{graphicx}

\newcommand{\executeiffilenewer}[3]{%
\ifnum\pdfstrcmp{\pdffilemoddate{#1}}%
{\pdffilemoddate{#2}}>0%
{\immediate\write18{#3}}\fi%
}

\theoremstyle{plain}
\newtheorem{proposition}{Proposition}[section]
\newtheorem{theorem}{Theorem}[section]
\newtheorem{lemma}{Lemma}[section]
\newtheorem{corollary}{Corollary}[section]

\newtheorem{conjecture}{Conjecture}[section]

\newtheorem{remark}{Remark}[section]

\newtheorem{problem}{Problem}[section]

\input amssym.def
\newsymbol\rtimes 226F
\newfont{\nset}{msbm10}

\def\e{\mbox{\boldmath $e$}}

\def\m{\mbox{\boldmath $m$}}

\def\vecv{{\mbox {\boldmath $v$}}}

\def\0{\mbox{\boldmath $0$}}
\def\1{\mbox{\boldmath $1$}}

\def\mod{\mathop{\rm mod }\nolimits}

\DeclareMathOperator{\Col}{Col}

\begin{document}

\title{On measures of edge-uncolorability of cubic graphs:\\
 A brief survey and some new results}
%Complexity measures of edge-\\uncolorability in cubic graphs}

\author{M.A. Fiol$^a$, G. Mazzuoccolo$^b$, E. Steffen$^c$\\
\\ {\small $^a$Barcelona Graduate School of Mathematics}
\\ {\small and}
\\ {\small Departament de Matem\`atiques}
\\ {\small Universitat Polit\`ecnica de Catalunya}
\\ {\small Jordi Girona 1-3 , M\`odul C3, Campus Nord }
\\ {\small 08034 Barcelona, Catalonia.}
\\
\\ {\small $^b$Dipartimento di Informatica}
\\ {\small Universit\'a di Verona}
\\ {\small Strada le Grazie 15, 37134 Verona, Italy.}
\\
\\ {\small $^c$Paderborn Center for Advanced Studies }
\\ {\small and}
\\ {\small Institute for Mathematics}
\\ {\small Universit\"at Paderborn}
\\ {\small F\"urstenallee 11, D-33102 Paderborn, Germany.}\\
\\  {\small E-mails: {\tt miguel.angel.fiol@upc.edu,}}\\
 {\small {\tt giuseppe.mazzuoccolo@univr.it, es@upb.de}}
  }

%\date{}
\maketitle
\begin{abstract}
There are many hard conjectures in graph theory, like Tutte's 5-flow conjecture, and the 5-cycle double cover conjecture, which would be true in general if they would be true for cubic graphs. Since most of them are trivially true for
3-edge-colorable cubic graphs, cubic graphs which are not 3-edge-colorable, often called {\em snarks}, play a key role in this context. Here, we survey parameters measuring
how far apart a non 3-edge-colorable graph is from being 3-edge-colorable. We study their interrelation and prove some new results.
Besides getting new insight into the structure of snarks, we show that such  measures give partial results with respect to these important conjectures. The paper closes with a list of
open problems and conjectures.
%In this paper we survey the different  measures of edge-uncolorability in cubic graphs that have been defined in the literature.%, and also consider some new ones.
%We discuss their  similarities and differences, and related results in the classification of non-edge-colorable graphs, mainly snarks (the case of cubic graphs). We prove that some of such measures are equivalent. Besides colorings,  we comment upon flows  (e.g., Tutte's 5-flow conjecture), factors (e.g., Berge's conjecture, Fulkerson's conjecture) and other structural parameters, and relate them to each other.
%We end by showing that, besides the objective to gain new insight into the structure of snarks, such  measures give partial results with respect to these important conjectures.
\end{abstract}

\noindent \textbf{Mathematics Subject Classifications:}  05C15, 05C21, 05C70, 05C75.

\noindent \textbf{Keywords:} Cubic graph; Tait coloring; snark; Boole coloring;  Berge's conjecture; Tutte's 5-flow conjecture; Fulkerson's Conjecture.

\section{Introduction}

We begin by commenting upon the main motivation of this paper.

\subsection{Motivation}

There are many hard problems in graph theory which can be solved in the general case if they can be solved for cubic graphs. Examples  of such problems are the 4-color-problem (now a theorem), problems concerning cycle- and matching-covers,
surface embeddings or flow-problems on graphs. Most of these problems are easy to solve for 3-edge-colorable cubic graphs.
By Vizing's theorems \cite{v64,v65}, or Johnson's \cite{j66} for the case $\Delta=3$, a cubic graph is either 3- or 4-edge-colorable.
Graphs with chromatic index greater than  their maximum degree are often called {\em class 2} graphs and {\em class 1} otherwise.

Bridgeless cubic class 2 graphs with cycle separating 2- or 3-cuts, or 4-circuits, can be constructed from smaller cubic class 2 graphs by some easy operations. Such substructures are excluded in possible minimal counterexamples for the most of the problems.
Thus, possible minimal counterexamples are asked to be cyclically 4-edge-connected class 2 cubic  graphs with girth at least 5 (see, for instance, Chetwynd and Wilson\cite{cw81}, Isaacs \cite{i75}, and Watkins \cite{w83}).
Such graphs were called `snarks' by Gardner \cite{g76}, who borrowed the name
from a nonsense poem by the famous English author Lewis Carroll~\cite{Ca74}.
However,  the decomposition results given by Cameron, Chetwynd, and Watkins \cite{ccw87} and  Haj\'os \cite{ha61} showed that this notion of non-triviality may not be appropriate. Thus, some authors adopt the most simple definition stating that a snark is a class 2 bridgeless cubic graph. Moreover, we remark that, in some cases (for instance, in Berge's and Fan-Raspaud's conjectures), it is not known if a minimal counterexample should satisfy the above strong definition of snark.
For the sake of clarity, throughout the paper we will call {\em snark} to a class 2 cubic graph satisfying the strong definition, and otherwise we will speak of a
{\em class 2 bridgeless cubic graph}.

\subsection{Historical remarks}
To the authors' knowledge, the history of the hunting of (non trivial) snarks may be summarized as follows. In 1973 only four snarks were known, the earliest one being the ubiquitous Petersen graph $P$ \cite{p81}. The other three, on  18, 210, and  50  vertices, were found by Blanu\v{s}a \cite{b46},  Descartes (a pseudonymous of Tutte) \cite{d48}, and Szekeres \cite{s73} respectively. Then, quoting Chetwynd and Wilson \cite{ccw87}, ``In 1975 the art of snark hunting underwent a dramatic change when Isaacs \cite{i75} described two infinite families of snarks." One of these families, called the {\em BDS class}, included all (three) snarks previously known. In fact, this family is  based on a construction also discovered independently by  Adelson-Velski and  Titov in \cite{avt74}. The members of the other family are the so-called {\em flower snarks}. They were also found independently by Grinberg in 1972 \cite{go90}, although he never published his work. In \cite{j73}, Jakobsen proposed a method, based on the well-known Haj\'os-union \cite{ha61}, to construct class 2 graphs. As it was pointed out by Goldberg in \cite{go81}, some snarks of the BDS class can also be obtained by using this approach.
Later, Isaacs \cite{i76} described two new infinite sets of snarks found by Loupekhine.

In 1979, Fiol \cite{f79} proposed a new method of generating snarks, based on Boolean algebra. This method led to a new characterization of the BDS class, and also to a significant enlargement of it. For instance, Loupekhine's snarks \cite{i76} and most of the Goldberg's snarks \cite{go79,go81}  can be viewed as members of this class.
This approach is based on the interpretation of certain cubic graphs as logic circuits and, hence, relates the Tait coloring of such graphs with the SAT problem. Hence, such a method implicitly contains the result of Holyer \cite{hol81} proving that the problem of determining the chromatic index of a cubic graph is NP-complete.
In the same work \cite{f79}, infinitely many snarks of another family, called by Isaacs  the {\em Q class}, were also given. Apart from the Petersen graph $P$ and the flower snark  $J_5$ , in  \cite{i75} Isaacs had given only one further snark of this class: {\em the double star graph}. The graphs of this class are all cyclically  5-edge-connected. Subsequently, Cameron, Chetwynd, and Watkins \cite{ccw87} gave a method to construct new snarks belonging to such a family.
Other constructions of snarks, most of them belonging to the BDS class, have been proposed by several authors. See, for instance, the papers of Celmins and Swart \cite{csxx}, Fouquet,  Jolivet, and Rivi\`ere  \cite{fjr81}, and Watkins  \cite{w83}.

In \cite{Kochol_96} Kochol used a method, that he called {\em superposition}, to construct snarks with some specific structure (see an overview in \cite{ko09}).
Kochol used the concept of flows (see Section \ref{sec:flows}) to define superposition. However, the approaches via flows and logical circuits are equivalent since both can be described
as 3-edge-colorings with colors from the Klein four group. Thus, Kochol's {\em supervertices} and {\em superedges} are special instances of the {\em multisets}
used by Fiol in \cite{f79,f95}. More generally, in these papers there is a method to construct  multisets representing all logic gates of a circuit, which allows us to use `superposition' in a broader context (see Subsection 2.3).
By using this method, Kochol disproved two old conjectures on snarks. First, he constructed snarks of girth greater than 6,
disproving a conjecture of Jaeger and Swart \cite{jsw80}. Second, in \cite{ko09_2} Kochol disproved a conjecture of Gr\"{u}nbaum \cite{Gr68}, that snarks do not have polyhedral embeddings into orientable surfaces.

Later on, stronger criteria of non-triviality and reductions/constructions of snarks are considered in  many papers: Chetwynd and Hilton \cite{ch88},
Nedela and \v{S}koviera \cite{ns96},
Brinkmann and Steffen \cite{bs98},
Steffen \cite{s98,s01},
Gr\"unewald and Steffen \cite{gs99},
M\'a\v{c}ajov\'a and ~\v{S}koviera \cite{ms06},
Chladn\'y and \v{S}koviera \cite{cs10},
Karab\'a\v{s}, M\'a\v{c}ajov\'a, and Nedela  \cite{kmn13}, and Sinclair \cite{sinclair_97}, among others.
However, non of them leads to recursive construction of all snarks, as it is known for 3-connected graphs \cite{b91,t61}. Intuitively, a snark which is not reducible to a 3-edge-colorable graph seems to be more complicated or of higher complexity than a snark which is reducible to a 3-edge-colorable cubic graph.

One major difficulty in proving theorems for snarks is to find/define appropriate structural parameters for a proof. This leads to the study of invariants that measure `how far apart' the graph is from being 3-edge-colorable. Isaacs \cite{i75} called cubic class 2 graphs {\em uncolorable}. Hence, these invariants are sometimes called measures of edge-uncolorability in the literature. On one side, these parameters can give new insight into
the structure of bridgeless cubic class 2 graphs, on the other side they allow to prove partial results for some hard conjectures.
%We are are going to study the following measures of edge-uncolorability:

\subsection{Some strong conjectures}
The formulation of the 4-Color-Theorem in terms of edge-colorings of bridgeless planar cubic graphs is due to Tait \cite{Tait} (1880).
Tutte generalized the ideas of Tait when he introduced nowhere-zero flows on graphs \cite{t54,t56,t66}. He conjectured in 1954 that every bridgeless graph has a nowhere-zero 5-flow. This conjecture is equivalent to its restriction to cubic graphs.

In a recent paper, Brinkmann, Goedgebeur, H\"agglund, and Markstr\"om \cite{Gunnar_2011_paper} generated a list of all snarks up to 36 vertices and tested whether some conjectures are true
for these graphs. They disproved some conjectures. However, the most prominent ones are true for these graphs.

\begin{conjecture} [5-Flow Conjecture, Tutte \cite{t54}] \label{5flow_conj}
Every bridgeless graph has a nowhere-zero 5-flow.
\end{conjecture}

The following conjecture is attributed to Berge (unpublished, see e.g.~\cite{{fouqvan09}, {gm11}}).

\begin{conjecture} [Berge Conjecture] \label{Berge_Conjecture}
Every bridgeless cubic graph has five perfect matchings such that every edge is in at least one of them.
\end{conjecture}

Conjecture \ref{Berge_Conjecture} is true if the following conjecture is true, which is also attributed to Berge in \cite{Seymour_79}.
This conjecture was first published in a paper by Fulkerson \cite{Fulkerson_1971}.

\begin{conjecture} [Berge-Fulkerson Conjecture \cite{Fulkerson_1971}] \label{Fulkerson}
Every bridgeless cubic graph has six perfect matchings such that every edge is in precisely two of them.
\end{conjecture}

Mazzuoccolo \cite{gm11} proved that Conjectures \ref{Berge_Conjecture} and  \ref{Fulkerson} are equivalent.
The following conjecture of Fan and Raspaud is true if Conjecture \ref{Fulkerson} is true.

\begin{conjecture} [Fan-Raspaud Conjecture \cite{Fan_Raspaud_94}] \label{Fan_Raspaud}
Every bridgeless cubic graph has three 1-factors such that no edge is in each of them.
\end{conjecture}

A {\em cycle} is a 2-regular graph, and its components are called {\em circuits}.
A {\em 5-cycle double cover}  of a graph $G$ is a set of five cycles, such that every edge is in precisely two of them.
The following conjecture was stated by Celmins and Preissmann independently.

\begin{conjecture} [5-Cycle-Double-Cover Conjecture, see \cite{Zhang_book}] \label{5-cdcc}
Every bridgeless graph has a 5-cycle double cover.
\end{conjecture}

\subsection{Basic notation and first parameters of edge-uncolorability}

Let $G$ be a graph with vertex set $V=V(G)$ and edge set $E=E(G)$.  Let $deg(v)$ be the  degree of vertex $v\in V$, and let  $\Delta=\Delta(G)$ denote the maximum degree of $G$. A mapping $c : E(G) \longrightarrow {\cal C}=\{1, 2, \dots, k\}$ is a {\em $k$-edge-coloring} of $G$.
If $c(e) \not= c(e')$ for any two adjacent edges $e$ and $e'$,
then $c$ is a {\em proper $k$-edge-coloring} of $G$. The {\em chromatic index} of $G$, denoted by $\chi'(G)$, is the minimum
integer $k$ such that $G$ has a proper $k$-edge-coloring.
By the well-known theorem of Vizing \cite{v64}, $\chi'(G)$ must be either $\Delta(G)$ or $\Delta(G)+1$ (see Johnson \cite{j66} for the case $\Delta=3$).
In the former case, $G$ is said to be {\em class 1}. Otherwise, $G$ is said to be {\em class 2}.
In  \cite{fv13}, Fiol and Vilaltella proposed a simple, but empirically efficient, heuristic algorithm for edge-coloring of graphs, which is based on the displacement of conflicting vertices.
If $c$ is a proper edge-coloring coloring
of $G$, then we say that $c^{-1}(i)$ is a {\em color class}. Clearly, $c^{-1}(i)$ is  a matching in $G$. A graph $G$ with $\Delta(G)=3$ is called {\em subcubic} and, if $G$ is regular, then $G$ is a {\em cubic} graph.
A proper $3$-edge-coloring of a cubic graph is also refered to as a {\em Tait coloring}.

The following parameters to measure how far apart a graph is from being edge-colorable were first defined (they are listed in the historical order they were proposed). As commented above, these are used both, to gain information about the structure of class 2 cubic graphs, and to obtain partial results on the aforementioned conjectures.
\begin{itemize}
\item
(Fiol \cite{f79,f95})
$d(G)$: The {\em edge-coloring degree} $d(G)$ is the minimum number of conflicting vertices  (i.e., with some incident edges having the same color) in a  3-edge-coloring of $G$.
\item
(Huck and Kochol \cite{hk95})
$\omega'(G)$: The {\em weak oddness} $\omega'(G)$ is the minimum number of odd components of an even factor $F$ of $G$. Note, that $F$ may contain vertices of degree 0.
\item
(Huck and Kochol \cite{hk95})
 $\omega(G)$: The {\em oddness} $\omega(G)$ is the smallest number of odd components in a 2-factor of $G$.
\item (Steffen \cite{s98,s04})
$r(G)$: The {\em resistence} $r(G)$ is the minimum cardinality of a color class of a proper 4-edge-coloring of $G$. Clearly, this is precisely the minimum number of edges whose removal yields a 3-edge-colorable graph.
We will say {\em minimal} a proper $4$-edge coloring of $G$ with a color class of cardinality $r(G)$.
This parameter is called {\em color number} and denoted by $c(G)$ in \cite{s98}.
\item
(Steffen \cite{s98}, Kochol \cite{ko11}, Mkrtchyan and Steffen \cite{ms12})
$\rho(G)\equiv r_v(G)$ is the minimum number of vertices to be deleted from $G$ so that the resulting graph has a proper $3$-edge-coloring.
\end{itemize}

Other  (more recent) measures or concepts related to edge-uncolorability, considered in subsequent sections of the paper, are:  reduction and decomposition of snarks (Nedela and \v{S}koviera \cite{ns96}); maximum 2-edge- and 3-edge-colorable subgraphs of snarks (Steffen \cite{s04}); excessive index (Bonisoli and Cariolaro  \cite{bc07}); $\mu_3$ (Steffen \cite{s15});
nowhere-zero flows (Tutte \cite{t54}); flow resistance (introduced in this paper); oddness and resistance ratios; etc. Moreover, some measures of edge-uncolorability are closely related to some types of reductions and vice versa, see for example Nedela and \v{S}koviera \cite{ns96} and Steffen \cite{s98}.

In this work we give a survey on results on the different  measures of edge-uncolorability in cubic graphs, and give some new results on their relation to each other.
We also discuss their  similarities and differences, and related results in the attempt of a classification of non-edge-colorable graphs, mainly snarks (the case of cubic graphs).
The paper is organized as follows. We distinguish coloring, flow and structural parameters in the next sections and relate them to each other.
For general terminology and notations on graphs, see for instance, Bondy and Murty \cite{bm76}, Diestel \cite{Di97}, or Chartrand, Lesniak and Zhang \cite{clz16}.  For results on edge-coloring, see  e.g. Fiorini and Wilson \cite{fw77} or
Stiebitz,  Scheide, Toft, and Favrholdt \cite{sstf12}.

%%%%%%%%%%%%%%%%% Edge-colorings

%%%%%%%%%%%%%%%%%%%%%%%%%%%%%%%%%%%%
\section{Basic results and multipoles}
\label{sec:defi}
We begin by considering the first measures of uncolorability that were defined in the Introduction.
Basically, they concern with the concepts of conflicting vertex, conflicting zone, and oddness. Since their definition does not depend on the regularity of the graph, now we will focus on subcubic graphs. Let $c$ be a 3-edge-coloring of a graph $G$ with $\Delta(G)=3$. For $v \in V(G)$, let $c(v)$ be the set of colors that appear at $v$. Vertex $v$ is a {\em conflicting vertex} if $|c(v)| < deg(v)$,
and it is a {\em normalized} conflicting vertex if $|c(v)| = deg(v)-1$, i.e., there is precisely one color that appears twice at $v$.
A {\em conflicting zone} is a subgraph containing some conflicting vertex for any 3-edge-coloring of $G$. A {\em conflicting edge-cut} is a set of edges of $G$ that separates two conflicting zones.

%\subsection{Some basic results}
As shown in the following proposition,  all the first defined
parameters, excepting the oddness and weak oddness, coincide in the case of graphs with maximum degree $3$.

\begin{theorem}
For every subcubic graph $G$ the following hold:
\begin{itemize}
\item[$(i)$]
The graph $G$ has a proper 3-edge-coloring if and only if $d(G)=r(G)=\rho(G)=\omega(G)=\omega'(G)=0$.
\item[$(ii)$]
$d(G)=r(G)=\rho(G)$.
\item[$(iii)$]
A cubic graph $G$ is not edge-colorable (class 2) if and only if $d(G)\ge 2$.
\item[$(iv)$]
If $d(G)\le 1$, then $\omega(G)=d(G)$.
\item[$(v)$]
If $G$ is a cubic graph with $d(G)=2$, then $\omega(G)=d(G)$.
\item[$(vi)$]
In general, $\omega(G)\ge d(G)$, and both parameters can be arbitrarily far apart.
\end{itemize}
\end{theorem}
\begin{proof}
$(i)$ is a simple consequence of the definitions.
$(ii)$ is a direct consequence of the result by Fiol \cite[Theorem 2.1]{f95} stating that, if $d(G)=d$, then there is a 3-edge coloring with exactly $d$ normalized conflicting vertices.
This result was rediscovered by Kochol \cite{ko11} by proving that $\rho(G)=\sigma(G)$.
%for every graph $G$ with maximum degree 3.
Again, $(iii)$ is a consequence of Corollary 1.2 in \cite{f95}.
Concerning $(iv)$, if $d(G)=0$ the result is clear. Otherwise, if $d(G)=1$, the only conflicting vertex has degree 2 with edges having color, say, 1. Then, the Kempe chain 1-2 (or 1-3) gives an odd circuit and $\omega(G)=1$.
The result in $(v)$ is proved similarly by using Kempe chains, and it was given \cite[Theorem 2.2]{f95}.
Finally, the result in $(vi)$ was proved by Steffen in \cite[Theorem 2.3]{s04}.
\end{proof}

The following results are drawn from \cite{f79,f95,ff84}.

\begin{theorem}
\label{propo-f}
The following holds.
\begin{itemize}
\item[$(i)$]
Let $G$ be a class 2 graph. If $H\subset G$ is a conflicting zone with $d(H)=d\ge 2$, then there exists a conflicting zone $H'\subset H$ with $d(H')=d-1$.
\item[$(ii)$]
Let $S$ be a snark with $d(S)\ge 2$. Then, for every $d'=1,\ldots,d$, there exists a conflicting zone $H\subset G$ with $d(H)=d'$.
\item[$(iii)$]
Let $S$ be a snark with $d(S)>2$. Then $S$ contains a subdivision of a snark $S'$ with $d(S')=2$ ($S'$ is a minor of $S$).
\item[$(iv)$]
Let $S$ be a snark with $d(S)\ge 2$. Then,  its number of vertices satisfy
$|V(S)|\ge 10\left\lfloor (d+1)/2\right\rfloor$.
\end{itemize}
\end{theorem}

From Proposition \ref{propo-f}$(iv)$, the first author \cite{f79} also managed to prove the following theorem (the most difficult case was $N=16$ and it was rediscovered later by Fouquet \cite{fou82}):
\begin{theorem}
There is no snark with number of vertices $N\in\{12,14,16\}$.
\end{theorem}

Some of the aforementioned results were also considered for subcubic graphs by Fouquet and Vanherpe \cite{fouqvan13}, and  Rizzi \cite{Rizzi09}.

\subsection{Multipoles}

In the study of snarks it is useful to think of them as made up by joining two or more graphs with `dangling edges'. Following \cite{f79,f91}, we call these graphs multipoles. More precisely, a {\em multipole} or  {\em $m$-pole}  $M = (V, E, {\cal X})$  consists of a (finite) set of vertices  $V = V(M)$, a set of edges $E = E(M)$ or unordered pair of vertices, and a set ${\cal X} = {\cal X}(M)$, $|{\cal X}|= m$, whose elements $\epsilon$ are called {\em semiedges}. Each semiedge is associated either with one vertex or with another semiedge making up what we call an isolated edge.
For instance, Fig. \ref{7-pole}$(a)$ shows a $7$-pole with two free edges. Notice that a multipole can be disconnected or even be `empty', in the sense that it can have no vertices. The diagram of a generic $m$-pole is shown in Fig. \ref{7-pole}$(b)$.
%If $M$ is just a free edge, we denote it by \textbf{e}.

\begin{figure}[h]
\begin{center}
%\vskip-1cm
\includegraphics[width=12cm]{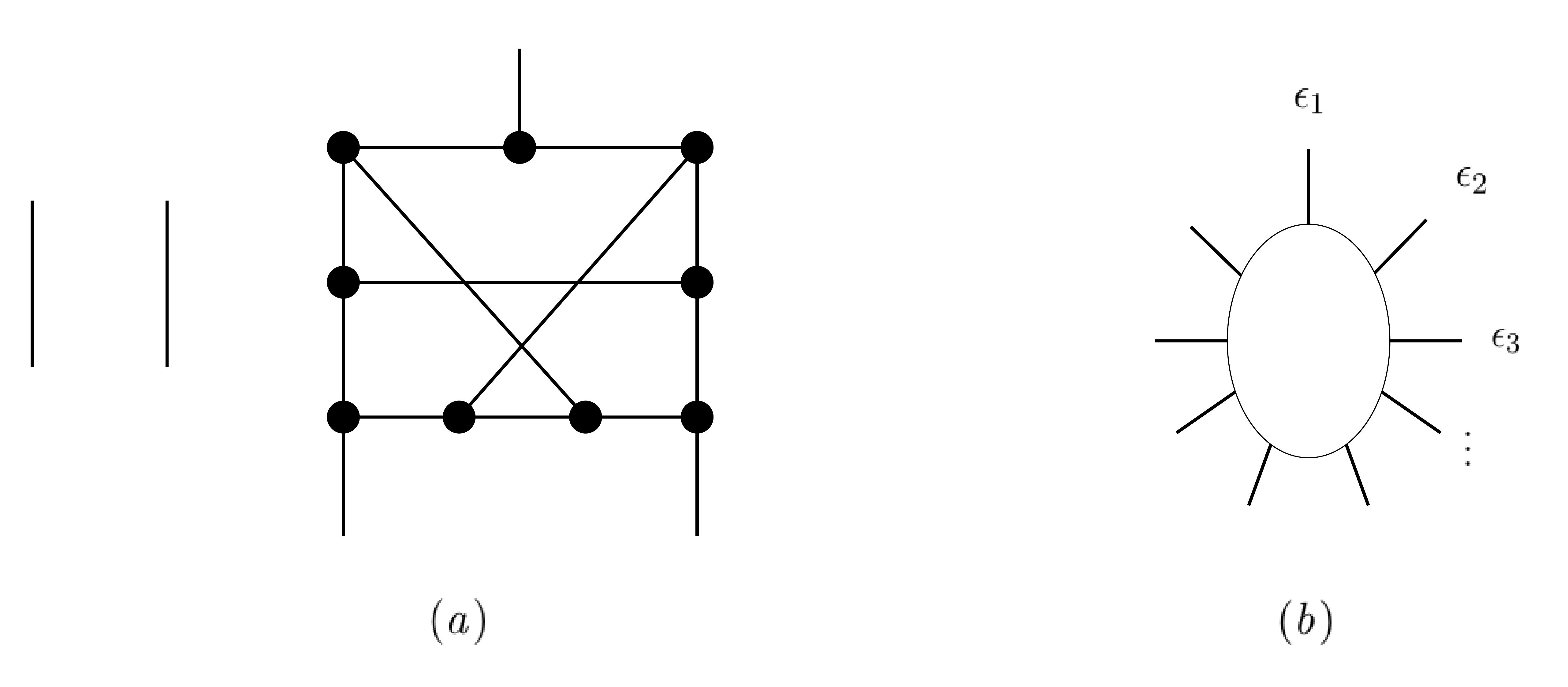}
\caption{A 7-pole with two free edges $(a)$ and a generic multipole $(b)$.}
\label{7-pole}
\end{center}
\end{figure}

The behavior of the semiedges is as expected.  For instance, if the semiedge   $\epsilon$ is associated with vertex $u$, we say that $\epsilon$ is {\em incident} to $u$. Then we write $\epsilon = (u)$ following Goldberg's notation \cite{go81}.  By joining the semiedges  $(u)$  and  $(v)$  we obtain the edge $(u, v)$. As for graphs, we define the {\em degree} of $u$, denoted by $deg(u)$, as the number of edges plus the number of semiedges incident to it. Throughout this paper, a multipole will be supposed to be cubic, i.e.,  $deg(u) = 3$ for all $u\in V$.

%Given a multipole $M$, we denote by $M^*$ the graph (with maximum degree 3) obtained from $M$  by leaving out all its semiedges.  Then, $M$  is said to be {\em contained} in a cubic graph $G$ if $M^*$ is a (proper) subdigraph of  $G$.  Notice that, in this case, $M$ can be thought of as being obtained from $G$ by cutting (in one or more points) some of its edges.

As expected, a {\em Tait coloring} of a $m$-pole $(V, E, {\cal E})$  is
an assignment of three colors to its edges and semiedges, that is, a mapping $\phi : E \cup {\cal E} \rightarrow {\cal C}=\{1,2,3\}$, such that the edges and/or semiedges incident to each vertex have different colors, and each isolated edge has both semiedges with the same color.
For example, Fig. \ref{colored-6-pole} shows a Tait coloring of a 6-pole. Note that the numbers of semiedges with the same color have the same parity. The following basic lemma states that this is always the case (see Izbicki \cite{Izbicki_66} and also Isaacs~\cite{i75}):

\begin{lemma}[{\bf The Parity Lemma}]
Let $M$ be a Tait colored $m$-pole with $m_i$ semiedges having color $i=1,2,3$. Then,
\begin{equation}
m_1\equiv m_2 \equiv m_3 \equiv m\qquad (\mod\ 2). \label{parity}
\end{equation}
\end{lemma}

\begin{figure}[h]
\begin{center}
%\vskip-1cm
\includegraphics[width=6cm]{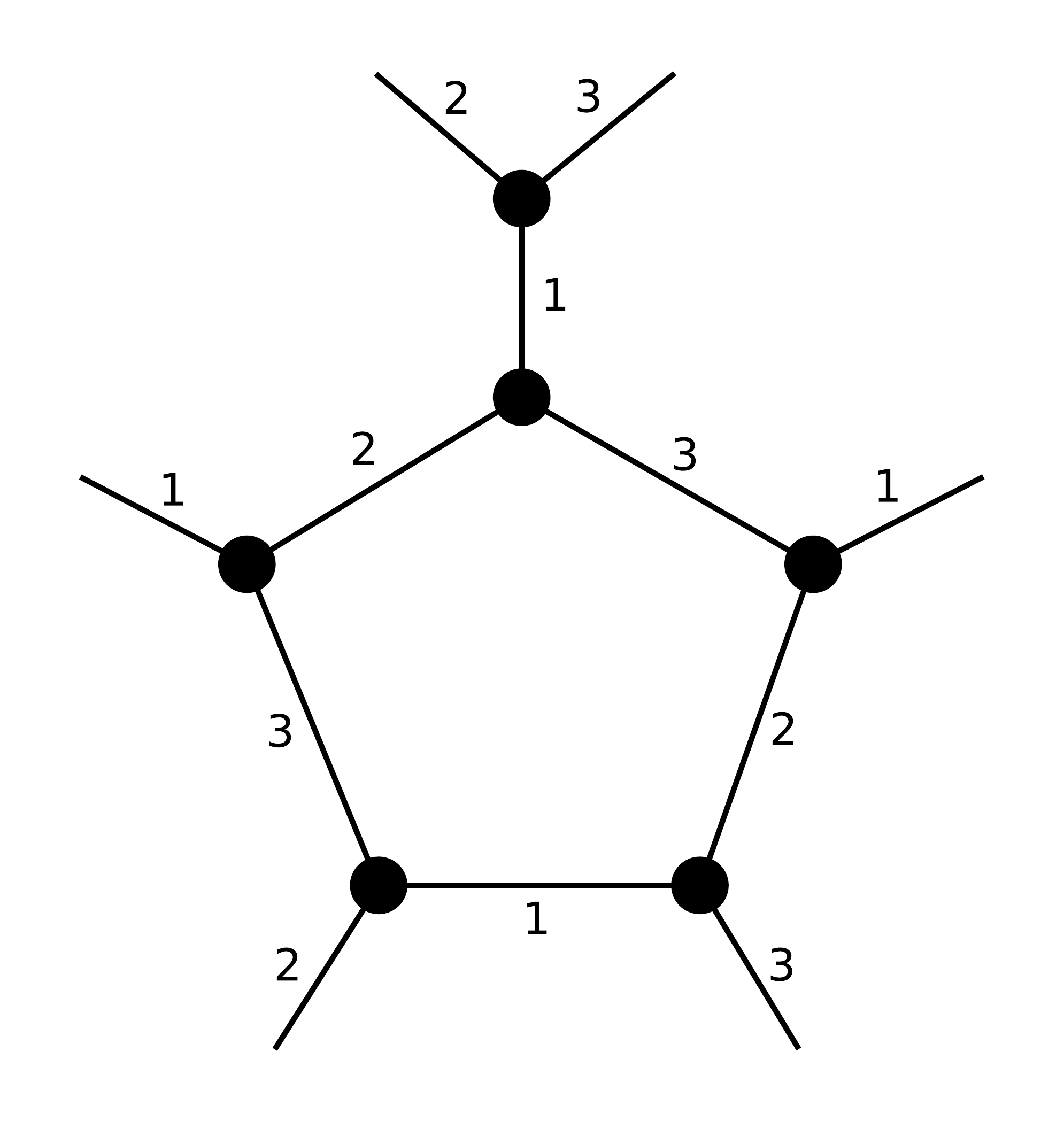}
\caption{A Tait coloring of a 6-pole.}
\label{colored-6-pole}
\end{center}
\end{figure}

This result has been used extensively in the literature on the subject. See, for instance,  Blanu\u{s}a \cite{b46}, Descartes \cite{d48},  Isaacs  \cite{i75}, and Goldberg \cite{go81}. Although in these references  isolated edges are not allowed, the proof is basically the same. A slightly more general version concerning Boole colorings will be proved in the next subsection.

Given an  $m$-pole  $M$  with semiedges  $\epsilon_1,\ldots,\epsilon_m$, we define its {\em set   $\Col(M)$ of semiedge colorings} as
$$
\Col(M) = \{(\phi(\epsilon_1),\phi(\epsilon_2),\ldots,\phi(\epsilon_m)):\mbox{\rm $\phi$ is a Tait coloring of $M$} \}.
$$
Note  that  $\Col(M)$ depends on the order in which the semiedges are considered. Thus, when referring to such a set we will implicitly assume that this ordering is given.

Of course, $\Col(M) =\emptyset$ if and only if  $M$  is not Tait colorable.  In this case it is trivial to obtain a class 2 graph from $M$. Indeed, we can either remove all its semiedges or join them properly in order to achieve regularity (using additional vertices if  necessary).  By  the parity lemma, the simplest example of non-Tait-colorable $m$-pole is when $m=1$, so that any cubic graph with a bridge is trivially class 2.

In the other extreme, we will say that $M$ is {\em color-complete} if $\Col(M)$ has maximum cardinality. In other words, $M$ is color-complete if it can be Tait colored so that its semiedges have any combination of colors satisfying the parity lemma. For instance, all Tait colorable 2-poles and 3-poles are   color-complete  because, according to \eqref{parity lemma}, the only possibilities, up to permutation of the colors, are $(\phi(\epsilon_1),\phi(\epsilon_2) = (a, a)$  and
$((\phi(\epsilon_1),\phi(\epsilon_2),(\phi(\epsilon_3)) = (a, b, c)$ respectively ---here, and henceforth, the letters $a$, $b$, $c$ stand for the colors 1, 2, 3 in any order. Clearly, the simplest color-complete 2-pole and 3-pole are respectively an isolated edge and a single vertex with 3 semiedges incident to it.  They will be denoted by  $\e$  and  $\vecv$ respectively.  Besides, a  color-complete 4-pole has four different values of $((\phi(\epsilon_1),\phi(\epsilon_2),(\phi(\epsilon_3)),(\phi(\epsilon_4))$. Namely, $(a, a, a, a)$, $(a, a, b, b)$, $(a, b, a, b)$  and  $(a, b, b, a)$.

A $m$-pole $M$ is said to be {\em reducible} when there exists an $m$-pole $N$ such that $|V(N)|<|V(M)|$ and $\Col(N)\subseteq \Col(M)$. Otherwise, we say that $M$ is {\em irreducible}. %Obviously, all minimal multipoles are irreducible.
This concept was first introduced by Fiol in \cite{f91}, where the following result was proved:

\begin{proposition}
Let $U$ be a snark. Then, for any integer $m\ge 1$ there exists a positive integer-valued function $v(m)$ such that any $m$-pole $M$ contained in $U$ with $|M|> v(m)$ is either not Tait colorable or reducible.
\end{proposition}

The known values of $v(m)$ are $v(2)=0$, $v(3)=1$ (both are trivial results), $v(4)=2$ (Goldberg \cite{go81}), and $v(5)=5$ (Cameron, Chetwynd, and Watkins \cite{ccw87}),  whereas its exact value is unknown for $m\ge 6$. However, in this case, Karab\'a\v{s}, M\'ac\v{a}jov\'a, and  Nedela \cite{kmn13} proved that $v(6)\ge 12$. These are much relevant questions in the decompositions of snarks. According to the Jaeger-Swart's conjecture \cite{jsw80}, every snark contains a cycle-separating edge-cut of size at most six. In that case, $v(6)$ would be the most interesting unknown value of $v(m)$.
Moreover, the above definition implies that any snark $U$ with a cutset of $m$  edges and  $|V(U)|>2v(m)$  can be `reduced' to another snark with fewer vertices. See \cite{ccw87} for the cases $m = 4,5$. More recently, Fiol and Vilaltella \cite{fv15} proved that the tree and cycle multipoles are irreducible and, as a byproduct, that $v(m)$ has a linear lower bound.

Let $M_1$ and  $M_2$  be two $m$-poles with semiedges  $\epsilon_i$ and $\zeta_i$,  $i=1,2,\ldots, m$, respectively, and assume that by joining  $\epsilon_i$ with $\zeta_i$ for all $i=1,2, \ldots , m$ we obtain the cubic graph $G$.  Then we will say that  $M_1$   and  $M_2$  are {\em complementary (with respect to $G$)}, or that $M_2$  is the {\em complement} of $M_1$, written   $M_2=M_1'$.
Moreover, the $m$-poles  $M_1$  and  $M_2$  are said to be  {\em color-disjoint} if $\Col(M_1) \cap \Col(M_2)=\emptyset$. In particular this is the case when one of the $m$-poles is not Tait colorable. The analysis and synthesis of snarks is based on the following straightforward result.

\begin{proposition}
Let  $M$  and  $M'$  be two complementary multipoles of a graph  $G$. Then
$G$ is a snark  if and only if $M$ and  $M'$  are color-disjoint.
\end{proposition}

Thus, the problem of constructing snarks can be reduced to the problem of finding pairs of color-disjoint multipoles. The main problem to proceed in this way is that, when the number of semiedges increases, the characterization of the set $\Col(M)$ becomes  more and more difficult. To overcome this drawback the idea is to group the semiedges in different sets, making the so-called {\em multisets}, and give a proper characterization of the `global' coloring of their elements, as we do in the next section.

\subsection{Boole colorings} \label{Boole_colorings}
The construction of cubic graphs which
cannot be Tait-colored leads to Boolean algebra, which is commonly
used in the study of logic circuits. To this end, the first author \cite{f79,f91}
introduced a generalization of the concept of `color', which describes in a
simple way the coloring (`$\0$' or `$\1$') of any set of edges
or, more abstractly, of any family ${\cal F}$ of $m$ colors chosen between
three different colors of ${\cal C}=\{1,2,3\}$, such that color $i\in {\cal C}$ appears
$m_i$ times. This situation can be represented by the coloring-vector
$\m=(m_1,m_2,m_3)$, where $m=m_1+m_2+m_3$. Then, we say that $\cal F$ has {\it
Boole-coloring} $\0$, denoted by $Bc({\cal F})=\0$, if
$$
m_1\equiv m_2\equiv m_3 \equiv m \quad (\mod 2),
$$
whereas $\cal F$ has {\it Boole-coloring} $\1$ (more specifically $\1_a$), denoted by  $Bc({\cal F})=\1$ (or $Bc({\cal F})=\1_a$),
if
$$
m_a+1\equiv m_b\equiv m_c\equiv m+1 \quad (\mod 2),
$$
where, as before, $a,b,c$ represent the colors $1,2,3$ in any order. See  \cite{ff84,f91} for more information.

From these definitions, the Boole-coloring of an edge $e\in E$
with color $c(e)=a\in\cal{C}$ is $Bc(e)=Bc(\{a\})=\1_a$, and the
Boole-coloring of a vertex $v\in V$, denoted by $Bc(v)$, is defined
as the Boole-coloring of its incident edges, which can have either different
or the same colors. In this context, it is worth noting the
following facts:
\begin{itemize}
\item[{\bf F1.}]
If $deg(v)=1$, then $Bc (v)=\1_{a}$ if and only if the incident edge
to vertex $v$ has color $a\in\cal{C}$.
\item[{\bf F2.}]
If $deg(v)=2$, then $Bc (v)=\0$ if both incident edges to
vertex $v$ have the same color, and $Bc (v)=\1$ if not.
\item[{\bf F3.}]
If $deg(v)=3$, then $Bc (v)=\0$ if and only if the three incident edges
to vertex $v$ have three different colors. Thus, in a
Tait coloring of a cubic graph, all its vertices have Boole-coloring $\0$.
\end{itemize}

\begin{table}
\begin{center}
\begin{tabular}{|c||c|c|c|c|}
\hline
$+$ & $\0$ & $\1_1$ & $\1_2$ & $\1_3$\\
\hline\hline
$\0$ & $\0$ & $\1_1$ & $\1_2$ & $\1_3$\\
\hline
$\1_1$ & $\1_1$ & $\0$ & $\1_3$ & $\1_2$\\
\hline
$\1_2$ & $\1_2$ & $\1_3$ & $\0$ & $\1_1$\\
\hline
$\1_3$ & $\1_3$ & $\1_2$ & $\1_1$ & $\0$\\
\hline
\end{tabular}
\caption{Klein's group of Boole-colorings.}
\label{tableKlein}
\end{center}
\end{table}

Moreover, a natural sum operation can be defined in the set ${\cal
B}=\{\0,\1_{1}$, $\1_{2},\1_{3}\}$ of Boole-colorings in the following
way: Given the colorings $X_1$ and $X_2$ represented, respectively, by the
coloring-vectors $\m_1=(m_{11},m_{12},m_{13})$ and
$\m_2=(m_{21},m_{22},m_{23})$, we define the sum $X=X_1+X_2$ as the
coloring represented by the coloring vector $\m=\m_1+\m_2$. Then,
$(\cal{B},+)$ is isomorphic to the Klein group, with $\0$ as
identity, $\1_{a}+\1_{a}=\0$, and $\1_{a}+\1_{b}=\1_{c}$; see
Table~\ref{tableKlein}.

Notice that, since every element coincides with its inverse,
$m\1_{a}=\1_{a}+\1_{a}+\stackrel{m}{\cdots}+\1_{a}$ is $\0$ if $m$
is even and $\1_{a}$ if $m$ is odd. From this simple fact, we can
imply the following result (see Fiol~\cite{f95}), which is very useful in the further
development of the theory, and it can be regarded as another version
of the Parity Lemma (see, e.g., Isaacs~\cite{i75} or Izbicki \cite{Izbicki_66})

\begin{lemma}
\label{parity lemma}
 Let $G$ be a subcubic graph with $n$ vertices having a
$3$-edge-coloring, such that $n_{i}$
vertices have Boole-coloring $\1_{i}$, for $i\in \cal{C}$, with $n'=n_1+n_2+n_3\le n$. Then,
\begin{equation}\label{parity-lemma}
n_{1} \equiv  n_{2} \equiv  n_{3} \equiv  n'\quad (\mod 2).
\end{equation}
\end{lemma}
\begin{proof}
Indeed, since the Boole-coloring of each vertex is the sum of the
Boole-colorings of its incident edges, and recalling that $\sum_{v\in V}deg(v)=2|E|$, we can write
$$
\sum_{v\in V}Bc (v) =
\sum^{3}_{i=1}n_{i}\1_{i}+(n-n')\0=\sum^{3}_{i=1}n_{i}\1_{i}=\sum^{}_{e\in
E}2Bc (e) = \0,
$$
but this equality is only satisfied if, for every $i\in \cal{C}$, $n_{i}\1_{i}=\0$ or
$n_{i}\1_{i}=\1_{i}$. Then, from
$n_{1}+n_{2}+n_{3}=n'$, we get the result.
\end{proof}

Note that, if $G$ is a cubic graph with a given 3-edge-coloring, then the result apply
with $n_i$ being the number of conflicting vertices of type $\1_i$, $i=1,2,3$.

As a direct consequence of the lemma, we also get the following:
\begin{corollary}
There is no edge-coloring of a graph $G$ having only one vertex with
Boole-coloring $\1$ $($and the other vertices with Boole-coloring $\0)$.
\end{corollary}

Similar results are obtained in the context of the resistance in \cite{s98}. There it is shown that the uncolored edges can be classified
as in Lemma \ref{parity lemma}. Analogously, there is no 4-edge-coloring of a cubic graph with a color class of cardinality 1.

As previously mentioned, the concept of Boole  colorings allows us to use the theory of Boolean algebra for the construction and characterization of infinite families of snarks. An example
is the family obtained by joining adequately an odd number of copies
of the multiset shown in Figure~\ref{fig:multipols} (left).
This structure behaves as a NOT gate of logic circuits
in the sense that, its edges and semi-edges having been Tait-colored,
the colorings $X_1$ and $X_2$ are conjugated one to each other, namely $X_2=\0$ (respectively, $X_2=\1$) if and only if $X_1=\1$ (respectively, $X_1=\0$). This is satisfied for any coloring of semi-edge $e$. Two examples of this fact are shown in  Figure~\ref{fig:multipols} (center and right). If, as previously stated,
we join an odd number of these multipoles in a circular configuration, adding some vertices to connect semi-edges $e$, any attempt at Tait coloring
will lead to a conflict, and hence the graph is a snark. An example with five multipoles can be seen in Figure~\ref{fig:flower-snark}.
This family of snarks, called \emph{flower snarks}, was proposed by
Loupekhine (see Isaacs~\cite{i76}). As commented, the first infinite families
of snarks were given by Isaacs~\cite{i75}, but they can also be obtained by using Boole-colorings. More details on this technique can
be found in Fiol~\cite{f91}.

%%%%%%%%%%%%%%%%%%%%%%%%%%%%%%%%%%%
\begin{figure}[t]
\begin{center}
\includegraphics[width=13.5cm]{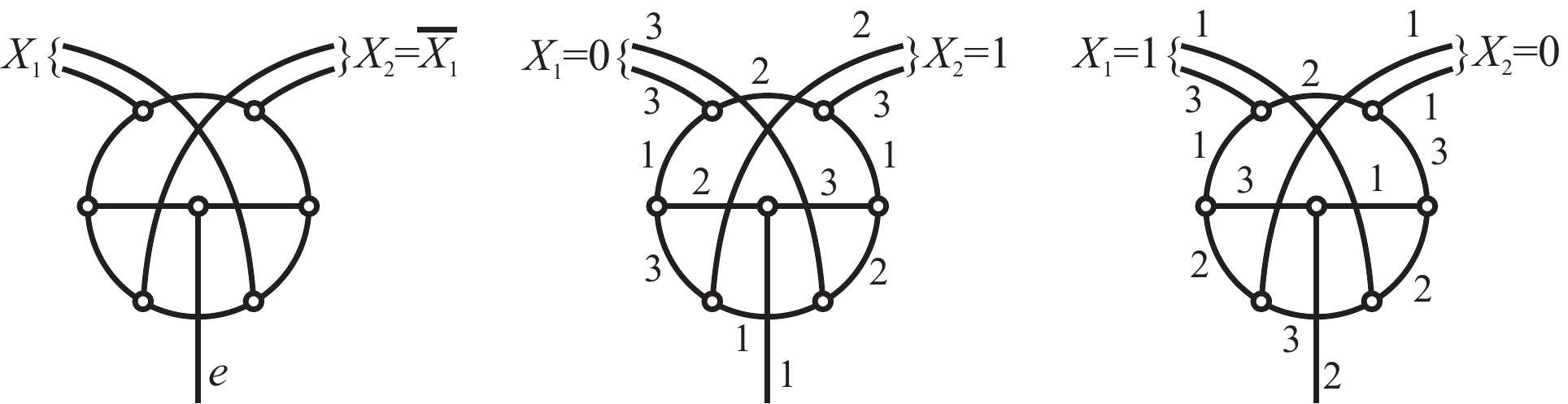}
\caption{Multipoles and the NOT gate.}
\label{fig:multipols}
\end{center}
\end{figure}
%%%%%%%%%%%%%%%%%%%%%%%%%%%%%%%%%%%
%%%%%%%%%%%%%%%%%%%%%%%%%%%%%%%%%%%
\begin{figure}[t]
\begin{center}
\includegraphics[width=5cm]{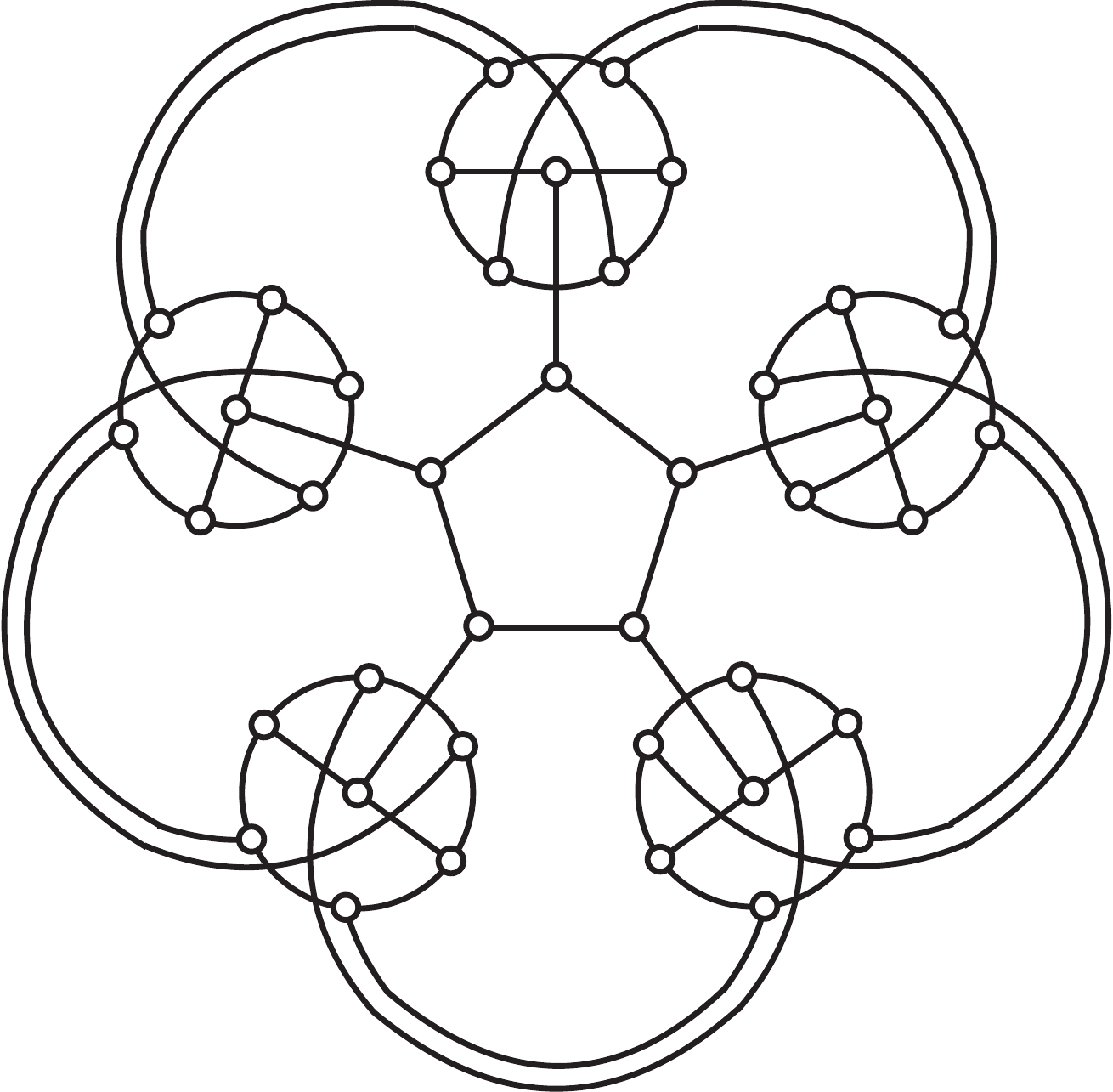}
\caption{A Loupekhine snark with 5 NOT gates.}
\label{fig:flower-snark}
\end{center}
\end{figure}
%%%%%%%%%%%%%%%%%%%%%%%%%%%%%%%%%%%

\subsection{Max 2- and 3-colorable subgraphs} \label{max_2-3-col_subgraphs}

Maximum $2$- and maximum $3$-edge-colorable subgraphs of cubic graphs were first studied by Albertson and Haas \cite{alberthaas96}.

The resistance measures the minimum number of uncolored edges in a 3-edge-coloring of a bridgeless cubic class 2 graph $G$.
We can ask a similar question with respect to 2-edge-colorable subgraphs of $G$, since if there is a 2-edge-colorable subgraph $H$
with $|E(H)| = \frac{2}{3}|E(G)|$, then $G$ is a class 1 graph.
Let $c_2 = \max \{|E(H)|: H \text{ is a subgraph of $G$ and } \chi'(H) = 2\}$, and $r_2(G) = \frac{2}{3}|E(G)| - c_2(G)$.
The following theorem shows that $r_2(G)$ is also a measure of edge-uncolorability.

\begin{theorem}[\cite{s04}]
If $G$ is a bridgeless cubic graph, then
\begin{itemize}
\item[$(i)$]
$r_2(G)=1$ if and anly if $r(G)=2$.
\item[$(ii)$]
 $\frac{1}{2}r(G) \leq r_2(G) \leq \min \{\frac{2}{3}r(G), \frac{1}{2} \omega(G)\}$, and the bounds are attained.
\end{itemize}
\end{theorem}

%%%%%%%%%%%%%%%%%%%%%%%%%%%% Factors
\section{Factors}

 In this section, we review some measures of edge-uncolorability of a cubic graph $G$ which depend on the properties of the sets of its $1$-factors and $2$-factors.

\subsection{1-factors}

We start from an unusual statement of Vizing's Theorem for cubic graphs:

{\it
The edge-set of every cubic graph can be written as a union of at most four of its matchings.
}

What happens if we would like to replace matchings with perfect matchings (i.e., 1-factors)? Can we prove an analogous theorem?

First of all, we must remark that previous question is only relevant in the class of bridgeless cubic graphs. Indeed, if a cubic graph $G$ has a bridge then some edge of $G$ does not lie in a 1-factor of $G$, hence we cannot obtain the edge-set of $G$ by union of $1$-factors.

Let $G$ be a bridgeless cubic graph, by Petersen's Theorem from 1891 \cite{p81}, $G$ has a 1-factor. Sch\"onberger \cite{Sch} refined this result by proving that
for every edge $e$, there is a 1-factor of $G$ which contains $e$. Sch\"onberger result implies that the edge-set of every bridgeless cubic graph $G$ can be obtained as a union of a finite set of $1$-factors. We denote by $\chi'_e(G)$ the minimum cardinality among all such sets of $1$-factors. The parameter $\chi'_e(G)$ is called {\em excessive index} in \cite{bc07} and {\em perfect matching index} in \cite{fouqvan09}.
Note that if a cubic graph has two disjoint 1-factors, then $\chi'(G)=\chi'_e(G)=3$.
Hence, $\chi'(G)=4$ if and only if any two 1-factors of $G$ have a non-empty intersection.

In an attempt of mimic Vizing's result, one could try to prove that $\chi'_e(G) \leq 4$ for every bridgeless cubic graph $G$, but such an attempt is guaranteed to fail. Indeed, the union of five disjoint $1$-factors of the Petersen graph is necessary to obtain its edge-set.

If Berge's conjecture (Conjecture \ref{Berge_Conjecture}) holds, then bridgeless cubic graphs which are not $3$-edge-colourable are divided into two classes according to they have excessive index $4$ or $5$. Hence, we can consider the excessive index as a measure of $3$-edge-uncolorability finer than chromatic index.
It is interesting that cubic graphs $G$ with $\chi'_e(G)=4$ share some properties with $3$-edge-colourable cubic graphs. For instance a shortest cycle cover (i.e., a set of cycles covering all the edges) of length $\frac{4}{3}|E(G)|$ \cite{s15}. Moreover, the 5-cycle double cover conjecture is true for such graphs (see Hou, Lai, and Zhang \cite{Hou_etal_2012}, and Steffen \cite{s15}). In other words, a possible counterexample for the 5-cycle double cover conjecture should be look for in the class of cubic graphs with excessive index $5$.
Abreu, Kaiser, Labbate, and Mazzuoccolo \cite{aklm16} suggest a relation between excessive index and circular flow number of a cubic graph (see Section \ref{sec:flows} for a definition). In particular, it is remarked that all known snarks with excessive index 5 have circular flow number at least 5. In other words, if a snark is critical with respect to Conjecture \ref{Berge_Conjecture}, then it seems to be critical also for the 5-flow conjecture. We believe that proving a relation between these two famous conjectures could be a very interesting result.

All previous arguments stress the fact that the class of snarks having excessive index five has a particular relevance and computational evidence shows that these snarks are quite rare. The smallest example, but the Petersen graph, has order $34$ and it was found by H\"agglund \cite{Hagglund_2012}. Starting from H\"{a}gglund's example, some infinite families of snarks with excessive index five have been recently constructed by Esperet and Mazzuoccolo \cite{em14}, Abreu, Kaiser, Labbate, and Mazzuoccolo \cite{aklm16}, and Chen \cite{chen16}.
Note that all such snarks are not cyclically $5$-edge-connected, it remains open the question about the existence of a cyclically $5$-edge-connected snark with excessive index 5 different from the Petersen graph (see Problem \ref{p_exc_index}). Such graphs had already been studied by Sinclair in \cite{sinclair_97}. He proves some decomposition results for graphs which cannot be covered by four perfect matchings. That paper is written in the language
of weights and cycle covers. The equivalence to the cover with perfect matchings follows from the fact that a snark $G$ can be covered by four
perfect matchings if and only if it has a cycle cover such that each edge is in one or in two circuits and the edges which are in two circuits form a
1-factor of $G$, see Theorem 3.5 of Hou, Lai, and Zhang \cite{Hou_etal_2012}.

%%%%%%%%%%%%%%%%%%%%%%%%%%%%%%%%%%%

Conjecture \ref{Berge_Conjecture} is largely open and it remains open even if we replace 5 with an arbitrary larger constant $k$. It is not hard to show, by using Edmonds' matching polytope theorem \cite{Edm65}, that the edges of any bridgeless cubic graph of order $n$ can be obtained as union of $ \log(n)$ perfect matchings. As far as we know, the best bound known (still logarithmic in the order of $G$) is the one given by Mazzuoccolo in \cite{m13} as a corollary of the technique introduced by Kaiser, Kr\'al and Norine in \cite{kkn06}.

%%%%%%%%%%%%%%%%%%%%%%%%%%%%%%%%%%%

In the remaining part of this section, we review known results about the cardinality and the structure of the union of a prescribed number of $1$-factors in a bridgeless cubic graph.
More precisely, let $G$ be a bridgeless cubic graph. Consider a list of $k$ 1-factors of $G$. Let $E_i$ be the
set of edges contained in precisely $i$ members of the $k$ 1-factors. Let $\mu_k(G)$ be the smallest $|E_0|$ over all lists of $k$ 1-factors of $G$ and set $m_k(G)= 1-\frac{\mu_k(G)}{|E(G)|}$, that is the maximum possible fraction of edges in a union of $k$ 1-factors of $G$.

We can restate some of the previous results about the excessive index in terms of these parameters.
For instance, a bridgeless cubic graph $G$ is $3$-edge-colorable if and only if $m_3(G)=1$ (or equivalently $\mu_3(G)=0$); and $m_5(G)=1$ for all $G$ is exactly Conjecture \ref{Berge_Conjecture}. Furthermore, if $G$ is a cubic bridgeless class 2 graph, then $\mu_3(G) \geq 3$, see \cite{s15}.

Kaiser, Kr\'al and Norine \cite{kkn06} proved that $m_2(G)\geq \frac{3}{5}$ and this result is the best possible, since the union of any two $1$-factors of the Petersen graph $P$ contains $9$ of the $15$ edges of the graph. It is also proved that $m_3(G)\geq \frac{27}{35}$, but it is conjectured that $m_3(G) \geq \frac{4}{5}=m_3(P)$ for every bridgeless cubic graph $G$.

Let $G$ be a cubic graph and $S_3$ be a list of three 1-factors $M_1, M_2, M_3$ of $G$. Let ${\cal M} = E_2 \cup E_3$, ${\cal U} = E_0$ and
$|{\cal U}| = k$. The edges of $E_0$ are also called the uncovered edges.
The  $k$-core of $G$ with respect to $S_3$ (or to $M_1, M_2, M_3$)  is the subgraph $G_c$ of $G$
which is induced by ${\cal M} \cup {\cal U}$; that is, $G_c = G[{\cal M} \cup {\cal U}]$.
If the value of $k$ is irrelevant, then we
say that $G_c$ is a core of $G$. If $M_1 = M_2 = M_3$, then $G_c = G$. A core $G_c$ is  proper if $G_c \not = G$.
If $G_c$ is a cycle, then we say that $G_c$ is a  cyclic core. In \cite{s15} it is shown that every bridgeless cubic graph has a proper core and therefore, every $\mu_3(G)$-core is proper. Cores of cubic graphs have bee studied by Jin and Steffen \cite{js15,s15}.

Let $\gamma_2(G) = \min \{|M_1 \cap M_2| : M_1 \text{ and } M_2 \text{ are 1-factors of } G\}$.
Then, $\mu_2(G) = \gamma_2(G) + \frac{1}{3}|E(G)|$.

\begin{theorem} \cite{js15} \label{weak_bound} Let $G$ be a bridgeless cubic graph. If $G$ is not 3-edge-colorable,
then $\omega(G) \leq 2\gamma_2(G)\leq \mu_3(G)-1$. Furthermore, if $G$ if $G$ has cyclic $\mu_3(G)$-core, then
$\gamma_2(G) \leq \frac{1}{3}\mu_3(G)$.
\end{theorem}

\begin{theorem} \cite{js15} \label{omega-mu3}
If $G$ is a bridgeless cubic graph, then $\omega(G)\leq \frac{2}{3}\mu_3(G)$.
\end{theorem}

In \cite{js15} it is also proved that there are graphs $G$ with $\omega(G) = \frac{2}{3}\mu_3(G)$, and that those graphs have very
specific structural properties. Indeed, if $G$ is a bridgeless cubic graph with $\omega(G) = \frac{2}{3}\mu_3(G)$, then it satisfies Conjecture \ref{Fan_Raspaud}.

%Oddness results for some conjectures: Fan-Raspaud \cite{Fan_Raspaud_94}.

\subsection{2-factors}

%A join of a graph $G$ is a set $J$ of edges such that the degrees of every vertex have the same parity in $G$ and $G[J]$. If there is no harm of
%confusion we use $J$ instead of $G[J]$.

Let $G$ be a bridgeless cubic graph. Obviously, $\omega'(G)\leq \omega(G)$. Note that an even factor of a bridgeless cubic graph is a spanning subgraph of $G$ having all vertices of degree either 2 or 0, that is, a union of circuits and isolated vertices.

In several papers, over the last few decades, weak oddness and oddness of a cubic graph appear as interchangeable definitions, implicitly assuming that they should be equal for every bridgeless cubic graph. But, the long standing discussion whether $\omega(G) = \omega'(G)$ for all bridgeless cubic graphs $G$ was recently finished by the following negative result of Lukot'ka and Maz\'{a}k.

\begin{theorem} [\cite{LM_2016}] \label{weak_oddness_vs_oddness} There exist a graph $G$ with $r(G)=12$, $\omega'(G) = 14$, and $\omega(G) = 16$.
\end{theorem}

Theorem \ref{weak_oddness_vs_oddness} gives rise to infinite families of cubic graphs where oddness and weak oddness differ. Moreover, it is observed in \cite{LM_2016} that there exist cubic graphs having arbitrarily large difference between oddness and weak oddness.

Here, we improve Theorem \ref{weak_oddness_vs_oddness}, showing an example of a cubic graph having $\omega'(G)=6$ and $\omega(G)=8$.

In order to give a general approach to the problem we introduce the following definitions:

A {\it minimal} $2$-factor (even factor) of a bridgeless cubic graph $G$ is a $2$-factor (even factor) of $G$ with the minimum number of odd circuits (components).
In other words, a minimal $2$-factor has $\omega(G)$ odd circuits and a minimal even factor has $\omega'(G)$ odd components.

Here we follow the terminology introduced by Espert and Mazzuoccolo in \cite{EM_15} for a standard operation on cubic graphs: given two cubic graphs $G$ and $H$ and two edges $xy$ in G and $uv$ in H, the \emph{glueing}, or 2-cut-connection, of $(G,xy)$ and $(H,uv)$ is the graph obtained from $G$ and $H$ by removing the edges $xy$ and $uv$, and adding the new edges $xu$ and $yv$.
In the resulting graph, we call these two new edges the \emph{clone edges} of $xy$ (or $uv$). Note that if $G$ and $H$ are cubic and bridgeless, then the resulting graph is also cubic and bridgeless.

In what follows $H$ will always be the Petersen graph $P$, which is arc-transitive. Hence, the choice of $uv$ and the order of each pair $(x, y)$ and $(u, v)$ are not relevant, so that we will simply say that we glue $P$ on the edge $xy$ of $G$.

\begin{lemma}\label{oddvsweakodd}
Let $(G,xy)$ be a pair where $G$ is a bridgeless cubic graph and $xy$ is an edge of $G$.
If $xy$ belongs to an odd circuit of a minimal even factor of $G$, {\boldmath $and$} $xy$ does not belong to a minimal $2$-factor of $G$, then the graph $G^*$ obtained by glueing a copy of the Petersen graph  on the edge $xy$ of $G$ has the following properties:
\begin{itemize}
 \item[$(i)$]
 $\omega(G^*)\ge \omega(G)+2$.
 \item[$(ii)$]
 $\omega'(G^*)=\omega'(G)$.
\end{itemize}
\end{lemma}

\begin{proof}
Denote by $uv$ the edge of $P$ used to perform the glueing (as already observed, the choice of $uv$ is not relevant since $P$ is arc-transitive).
Denote by $F'$ a minimal even factor of $G$ having an odd circuit $C'$ which contains the edge $xy$.
Select one of the minimal even factor of $P$ which consists of a circuit of length $9$ passing through $uv$ and an isolated vertex.
Consider the even factor of $G^*$ obtained glueing together $F'$ and the selected minimal even factor of $P$. The glueing of $C'$  and the $9$-circuit of $P$ (see Fig. \ref{fig:petersen_factors}) produces an even circuit of $G^*$, then this even factor of $G^*$ has the same number of odd components of $F'$, that is $\omega'(G^*)=\omega'(G)$.\\
Moreover, every $2$-factor of $P$ consists of two $5$-circuits. Hence, if every minimal $2$-factor of $G$ does not contain the edge $xy$, then the graph $G^*$ has at least two new odd components respect to $G$ in every of its $2$-factors, that is $\omega(G^*)\ge \omega(G)+2$ as claimed.
\end{proof}

Now, our aim is to costruct a pair $(G,xy)$ which satisfies the assumptions of previous lemma. Moreover, we would like to find $G$ such that their oddness and weak oddness be as small as possible. In particular, we will produce an example with oddness $6$ and we leave as an open problem the existence of an example with oddness $4$.

In order to construct such an example, we will implicitly use several times the two following properties of the graph $G^*$ described in Lemma \ref{oddvsweakodd}:
\begin{itemize}
\item[{\bf P1.}] If a circuit $C$ of a $2$-factor (even factor) $F$ of $G^*$ passes through the two clone edges of $xy$, then $F$ has exactly (at least) one odd component distinct from $C$ in $P$.
\item[{\bf P2.}] If a $2$-factor (even factor) $F$ does not contain a circuit $C$ passing through the clone edges of $xy$, then $F$ has exactly (at least) two odd components in $P$.
\end{itemize}

Finally, we recall in the next lemma some well-known properties of $2$-factors and circuits of the Petersen graph which are useful in the proof of our main result (see also Figure \ref{fig:petersen_factors}).

\begin{lemma}\label{petersen_prop}
Let $M=\{e_1,e_2,e_3\}$ be a set of three edges of the Petersen graph $P$, which induce a maximal matching of $P$. Then,
\begin{enumerate}
\item[$(i)$]
 Every $2$-factor $F$ of $P$ consists of two $5$-circuits $C_1$, $C_2$. Moreover, $|F \cap M|=2$, $|C_1 \cap M|=1$ and $|C_2 \cap M|=1$. \label{5cyclesPetersen}
\item[$(ii)$]
There exists a minimal even factor $F'$ of $P$ such that $M \subset F'$.  \label{9cyclePetersen}
\end{enumerate}
\end{lemma}

\begin{figure}
\centering
\includegraphics[width=8cm]{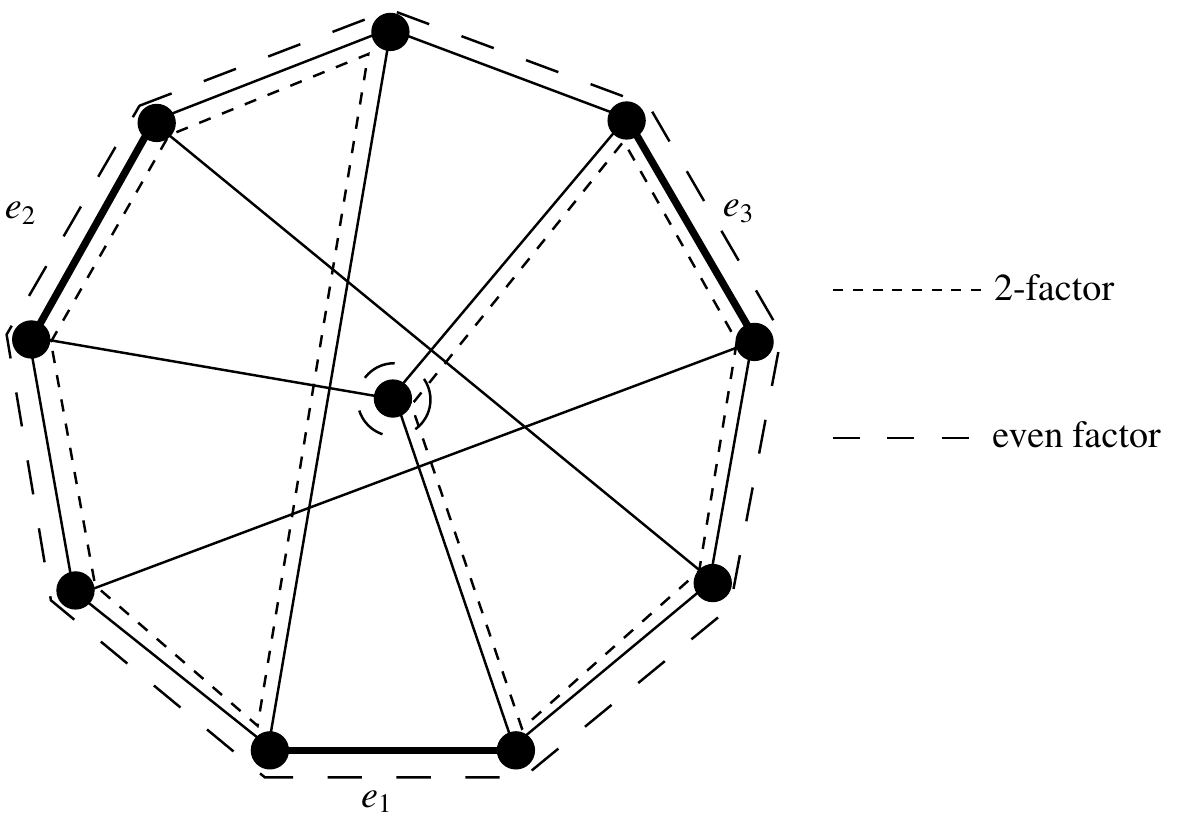}
\caption{A 2-factor $F$ and an even factor $F'$ of the Petersen graph which satisfy the conditions in Lemma \ref{petersen_prop}.}
\label{fig:petersen_factors}
\end{figure}

We denote by $K$ the graph obtained starting from the Petersen graph $P$ and glueing two further copies of the Petersen graph on each of two, say $e_2,e_3$, of the three edges of a maximal matching $M=\{e_1,e_2,e_3\}$ (see Fig.\ref{fig:K*}).
 %(This has to be understood as follows: given two copies $P_1$ and $P_2$ of the Petersen graph, \emph{glueing $P_1,P_2$ on the edge $e$ of $G$} means glueing $P_2$ on some clone edge of $e$ in the glueing of $P_1$ on the edge $e$ of $G$.)

\begin{remark}
 Every even factor and every $2$-factor either contains both edges or no edge of a given $2$-edge-cut. Since clone edges form a $2$-edge-cut, we can naturally reconstruct from each factor $F$ of $K$ the {\em underlying factor of $F$} in $P$.
 \end{remark}

\begin{figure}[h]
\centering
\vskip .4cm
\includegraphics[width=6cm]{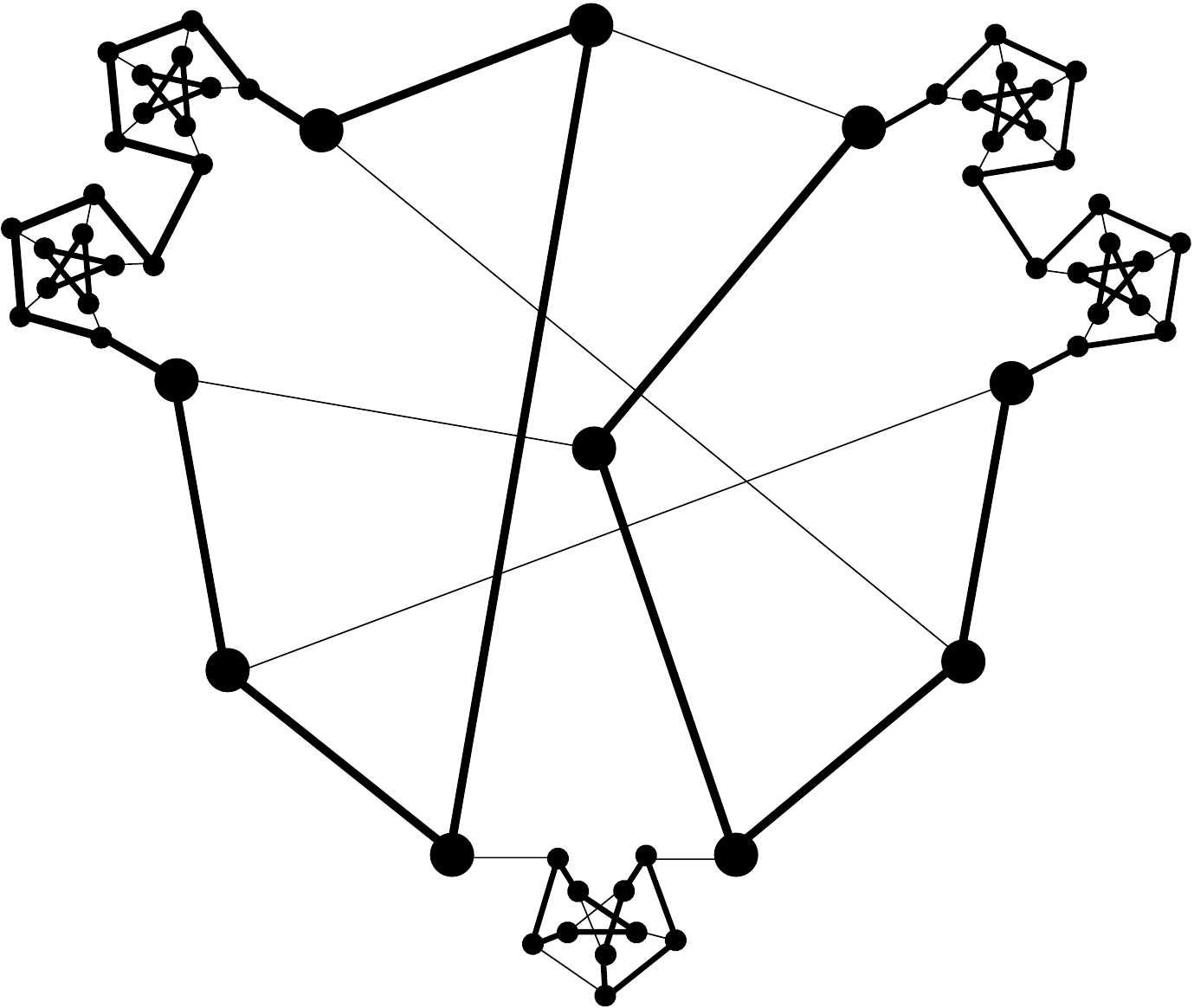}
\caption{The graph $K^*$ such that $\omega(K^*)=8$ and $\omega'(K^*)=6$. }
\label{fig:K*}
\end{figure}

Now, we are in position to prove our main result.
\begin{theorem} \label{improvement_by_GM}
Let $K^*$ be the graph obtained by glueing a copy of the Petersen graph on the edge $e_1$ of $K$ (see Fig. \ref{fig:K*}). Then, $\omega(K^*)=8$ and $\omega'(K^*)=6$.
\end{theorem}
\begin{proof}
We only need to prove that $\omega(K)=\omega'(K)=6$ and $(K,e_1)$ satisfies the assumptions of Lemma \ref{oddvsweakodd}.
Consider a $2$-factor $F$ of $K$. The underlying factor of $F$ in $P$ has two $5$-circuits.
If only one between $e_2$ and $e_3$, say $e_2$, lies in the underlying factor of $F$, then $F$ consists of eight odd circuits: two in each copy of the Petersen graph on $e_3$, one in each copy of the Petersen graph on $e_2$ and the two original odd circuits of $P$ (which still correspond to odd circuits in $K$).
On the other hand, if both $e_2$ and $e_3$ are edges of the two $5$-circuits, then $F$ has only six odd circuits: one in each of the four copies of the Petersen graph and the two original circuits of $P$ (which again correspond to odd circuits in $K$). Then, $\omega(K)=6$. Moreover, every minimal $2$-factor of $G$ does not contain the edge $e_1$ since point \ref{5cyclesPetersen} of previous lemma.
Since every even factor has at least one odd component in each of the four copies of the Petersen graph, and at least two odd components arising from the odd components in $P$, then also $\omega'(K)=6$ holds. Furthermore, it follows from property $P2$ that there exists an even factor of $K^*$ having an isolated vertex, four $5$-circuits (one in each copy of the Petersen) and a circuit of length $29$ passing through $e_1$.
Hence, by  Lemma \ref{oddvsweakodd}, the graph $K^*$ has $\omega'=6$ and $\omega\ge 8$. Finally, it is routine to check that, in fact, $\omega=8$ (see Fig. \ref{fig:K*}).
\end{proof}

As already observed, $\omega'(G)=2$ implies $\omega(G)=2$, then the remaining open question is whether $\omega'(G) = 4$ if and only if $\omega(G) = 4$ (see Problem \ref{conjec-oddness-res}). The next two theorem study the relation of the $\omega'(G)$ and $r(G)$, when $r(G) \in \{3,4\}$.
%recall Eckhard's Conjecture: \omega(G) \leq 2r(G)-2

\begin{theorem} \label{r=3}
Let $G$ be a bridgeless cubic graph. If $r(G) = 3$, then $\omega'(G) = 4$.
\end{theorem}

\begin{proof} Let $c$ be a minimal proper 4-edge-coloring of $G$ with $|c^{-1}(0)| = r(G)$.

Let $c^{-1}(0) = \{e_1, e_2, e_3\}$ and $e_i \in A_i$. Let $e_i = x_iy_i$. There is a $(2,3)$-chain $P_{(2,3)}(x_1,y_1)$
from $x_1$ to $y_1$ which starts with an edge of color 2 and ends with an edge of color 3.
The chains $P_{(1,3)}(x_2,y_2)$ and $P_{(1,2)}(x_3,y_3)$ are defined analogously. By Lemma 2.4 of \cite{s98} the three paths
$P_{(2,3)}(x_1,y_1)$, $P_{(1,3)}(x_2,y_2)$ and $P_{(1,2)}(x_3,y_3)$ are pairwise disjoint.

Furthermore, there are chains $P_{(1,3)}(x_1,y_3)$ and $P_{(2,3)}(x_2,x_3)$.

$(a)$ $P_{(1,3)}(x_1,y_3)$ and $P_{(2,3)}(x_2,x_3)$ are disjoint.

Interchange the colors on the chains; i.e., consider $P_{(3,1)}(x_1,y_3)$ and $P_{(3,2)}(x_2,x_3)$, to obtain a proper 4-edge-coloring $c'$ of
$G$. Then $e_3$ can be colored with color 3, and we still have a proper coloring. Hence $r(G) < 3$, a contradiction.

$(b)$ $P_{(1,3)}(x_1,y_3)$ and $P_{(2,3)}(x_2,x_3)$ are not disjoint.

Let $z$ be the first vertex of $P_{(1,3)}(x_1,y_3)$
which is incident to an edge of $P_{(2,3)}(x_2,x_3)$.
Then $P_{(1,3)}(x_1,z)$ is of odd length, i.e.,~its last edge is colored with color 1.

Let $x \in \{x_2,x_3\}$ and let $P_{(2,3)}(z,x)$ be the subpath of $P_{(2,3)}(x_2,x_3)$ which starts at $z$ with an edge of color 2.

$(b1)$ $x = x_2$. Replace $P_{(2,3)}(x_2,z)$ by $P_{(3,2)}(x_2,z)$,
and note that the coloring of the edges of $P_{(1,3)}(x_1,z)$ is unchanged. Replace  $P_{(1,3)}(x_1,z)$ by $P_{(3,1)}(x_1,z)$, to obtain a 4-edge-coloring
$\phi$ where all edges incident to $z$ are colored with color 3, for all $v \in V(G) - \{z\}$, the three edges incident to $v$ receive pairwise different colors.
Furthermore, $\phi^{-1}(0) = \{e_1, e_2, e_3\} = A_3$. Now, $G - \phi^{-1}(3)$ has (at most) 4 odd components.
Hence $\omega'(G) = 4$.

$(b2)$ $x = x_3$. Let $z'$ be the neighbor of $z$ in $P_{(1,3)}(x_1,y_3)$ such that $c(zz') = 3$, i.e.,~$zz' \in E(P_{(1,3)}(x_1,y_3)) \cap E(P_{(2,3)}(x_2,x_3))$.
Let
$P_{(2,3)}(x_2,z')$ be the subpath of $P_{(2,3)}(x_2,x_3)$. Obtain a proper 4-edge-coloring $\phi$ of $G$ by replacing
$P_{(2,3)}(x_2,z')$ by $P_{(3,2)}(x_2,z')$, and
$P_{(1,3)}(x_1,z)$ by $P_{(3,1)}(x_1,z)$, and color edge $zz'$ with color 0. Then $\phi^{-1}(0) = \{e_1,e_2,e_3, zz'\} = A_3$.
Now, $\phi^{-1}(3)$ is a 1-factor of $G$ and $G - \phi^{-1}(3)$ has (at most) 4 odd components.
Hence $\omega'(G) = 4$.
\end{proof}

\begin{theorem}
There is a bridgeless cubic graph $G$ with $r(G)=4$ and $\omega'(G)=\omega(G)=6$.
\end{theorem}
\begin{proof}

Let $P^-$ be the Petersen graph minus a vertex and $x, y, z$ be the three divalent vertices. For $i \in \{1,2,3\}$ there is minimal proper 4-edge-coloring $c_i$ of $P^-$ such that color 1 is missing at $x, z$ and color $i$ is missing at $y$. Let 0 be the fourth color and the edges of color 0 be the minimal color class. For $i \in \{1,2,3\}$ let
$P_i$ be a copy of $P^-$ with divalent vertices $x_i, y_i, z_i$ and a minimal proper 4-edge-coloring $c_i$.
Let $v$ be a vertex and add edges $y_iv$ (for each $i \in \{1,2,3\}$), $x_1z_{2}$, $x_2z_{3}$, $x_3z_{1}$ to obtain a 3-edge-connected cubic graph $H$ with 4-edge-coloring with three edges colored with color 0.
Since $H$ contains three pairwise disjoint 3-critical graphs it follows that $r(H) = 3$. (It is easy to see that $\omega'(H)=\omega(H)=4$.) Let $H_1$ and $H_2$ be two copies of $H$. Remove an edge $e_i$ of color 0 from $H_i[V(P_3)]$. Let $e_i=v_iw_i$ and add edges $v_1v_2$ and $w_1w_2$ to $H_1[V(P_3)]-e_1$ and $H_2[V(P_3)]-e_2$, to obtain a bridgeless cubic graph $G$ and a 4-edge-coloring of $G$ with minimal color class of cardinality 4. Since $G$ contains 4 pairwise disjoint 3-critical graphs it follows that $r(G) = 4$. It is easy to see that $\omega(G), \omega'(G) \leq 6$. Suppose to the contrary that $\omega'(G) = 4$. Then $G[V(H_1)]$ or $G[V(H_2)]$ (say $G[V(H_1)]$) contains at most two odd components of an even factor of $G$. But this implies that $H_1$ has an even factor with at most three odd components, a contradiction. We similarly deduce that $\omega(G)=6$.
\end{proof}

\section{Nowhere-zero flows}
\label{sec:flows}

An orientation $D$ of a graph $G$ is an assignment of a direction to each edge, and for $v \in V(G)$, $E^-(v)$  is the set of edges of $E(v)$ with head $v$ and $E^+(v)$ is the set of edges with tail $v$. The oriented graph is denoted by $D(G)$. If there is no harm of confusion we write $D$ instead of $D(G)$.

Let $A$ be an Abelian group. An {\em $A$-flow} $(D,\phi)$ on $G$ is an orientation $D$ of $G$ together with a function
$\phi: E(G) \rightarrow A$ such that
$$
\sum_{e \in E^+(v)}\phi(e) = \sum_{e \in E^-(v)}\phi(e), \textrm{ for all } v \in V(G).
$$
The set $\{e : \phi(e) \not= 0\}$ is the {\it support of} $(D,\phi)$ and it is denoted by $supp(D,\phi)$.
Furthermore, $(D,\phi)$ is a {\it nowhere-zero} $A$-flow, if $supp(D,\phi) = E(G)$.
If $A = \mathbb{R}$, the real numbers, then we say that $(D,\phi)$ is an {\it $r$-flow} if
$1 \leq |\phi(e)| \leq r-1$ or $\phi(e)=0$ for each $e \in E(G)$, and
$\sum_{e \in E^+(v)}\phi(e) = \sum_{e \in E^-(v)}\phi(e), \textrm{ for all } v \in V(G)$. Furthermore,
if $(D,\phi)$ only uses elements of $\mathbb{Z}$, the integers, then $(D,\phi)$ is an integer flow and it is called a $k$-flow. The following theorem of Tutte relates integer flows and group flows to each other.

\begin{theorem} [\cite{t54}]
\label{Tutte_Equiv_Flows}
Let $A$ be an Abelian group. A graph has a nowhere-zero $A$-flow if and only if it has a nowhere-zero $|A|$-flow.
\end{theorem}

If we reverse the orientation of an edge $e$ and replace the flow value by $-\phi(e)$, then we
obtain another nowhere-zero $A$-flow on $G$. Hence, if there exist an orientation of the edges of $G$
such that $G$ has a nowhere-zero $A$-flow, then $G$ has a nowhere-zero $A$-flow for any orientation.
Thus, the question for which values $r$ ($k$) a graph has a nowhere-zero $r$-flow (nowhere-zero $k$-flow) is a question about graphs,
not directed graphs.
The {\em circular flow number} of $G$ is $\inf\{ r | G \mbox{ has a nowhere-zero $r$-flow} \}$, and it is denoted by $F_c(G)$.
It is known  that $F_c(G)$ is always a minimum and that it is a rational number (see Goddyn, Tarsi, Zhang, and Cun-Quan \cite{GTZ_1998}).
Furthermore, if $F_c(G)$ is an integer, say $k$, then $G$ has an integer nowhere-zero $k$-flow, see Thm.~1.1 in \cite{s01}.
Let $F(G)$ be the smallest number $k$ such that $G$ admits a nowhere-zero $k$-flow. Clearly, $F_c(G) \leq F(G)$.
Indeed, Tutte conjectured that every bridgeless graph has a nowhere-zero 5-flow (Conjecture \ref{5flow_conj}).

Seymour \cite{Seymour_81} proved that every bridgeless graph has a nowhere-zero 6-flow, see \cite{drs17} for alternative proofs.
Conjecture \ref{5flow_conj} is equivalent to its restriction on cubic graphs. The following result shows that it suffices to prove it for snarks.

\begin{theorem} [Tutte \cite{t49,t54}]
\label{Tutte_character}
\begin{itemize}
\item[$(i)$]
A cubic graph $G$ is bipartite if and only if $F_c(G) = 3$.
\item[$(ii)$]
A cubic graph $G$ is class 1 if and only if $F_c(G) \leq 4$.
\end{itemize}
\end{theorem}

The following statement is a combination of results of the third author and Lukot'ka and \v{S}koviera.

\begin{theorem} [\cite{ls_11,s01_1}]
For every $s$ of the interval $(3,4)$, there is no cubic graph $G$ with $F_c(G) = s$, and for every $r \in \{3\} \cup [4,5]$, there is a cubic graph $H$ with $F_c(H)=r$.
\end{theorem}

Clearly, if $F_c(G) > 4$, then $G$ is a snark. However, it is not clear whether snarks with circular flow number close to 4 are
somehow less complex than snarks with circular flow number close to or equal to 5.
For instance, the Petersen graph has circular flow number 5, c.f.~\cite{s01_1}.
There are infinitely many snarks with circular flow number 5 (see M\'a\v{c}ajova and Raspaud \cite{mr06}, and Esperet, Mazzuoccolo, and Mkrtchyan \cite{emt16}), and also with further properties (see Abreu, Kaiser, Labbate, and Mazzuoccolo \cite{aklm16}).
From the results of Kochol \cite{Kochol_04,Kochol_10} and Mazzuoccolo and Steffen \cite{mast14} it follows that a minimal counterexample to Conjecture \ref{5flow_conj}
is cyclically 6-edge connected snark with girth at least 11 and with oddness at least 6. So far no such snark is known.

Tutte \cite{t69} conjectured that every graph $G$ with $F_c(G) > 4$ has a Petersen minor. This conjecture is still open, see Robertson, Seymour, and Thomas \cite{seymour_etal_2015}
for the current status of the work on that conjecture. However, in \cite{seymour_etal_2014} (by the same authors) it is shown that every cubic graph with girth at least 6 has a
Petersen minor. Hence, the existence of a Peterson minor is not a complexity measure for snarks with girth at least 6.

Jaeger and Swart \cite{jsw80} conjectured that if $G$ is a cyclically $k$-edge connected cubic graph and $k > 6$, then $F(G) \leq 4$.
Having in mind that the best known upper bound for the flow number is 6, the following result of Steffen \cite{s10} can also be seen as a first approximation to this conjecture.

\begin{theorem} [\cite{s10}]
Let $G$ be a cyclically $k$-edge connected cubic graph. If $k \geq \frac{5}{2}\omega(G)-3$, then $G$ has a nowhere-zero 5-flow.
\end{theorem}

As commented in the Introduction, Kochol \cite{Kochol_96} used the method of  superposition, which glues multisets (i.e., multipoles with different sets of terminals) together to  construct snarks.
As described in Section \ref{Boole_colorings}, such
constructions  strongly rely on the Klein four group $\mathbb{K}$ of Boole colorings.
Then, Kochol noted that it is useful to consider nowhere-zero flows in $\mathbb{K}$. Indeed, since every element of the Klein group is self inverse, we do not have to take care
of the orientation of the edges. Moreover, a nowhere-zero $\mathbb{K}$-flow on a cubic graph gives a 3-edge-coloring directly because of the isomorphism $\phi$, between the sets of colors and Boole-colorings, such that $\phi(i)=\1_i$ for $i \in \{1,2,3\}$.

\subsection{Flow resistance}

We will introduce a new parameter that measures how far apart a cubic graph is from having a nowhere-zero 4-flow. By Theorem \ref{Tutte_character}
this is also a complexity measure for snarks. Let $G$ be a cubic graphs and
$r_f(G) = \min \{|E(G)-supp(D,\phi)| : (D,\phi) \text{ is a $4$-flow on } G\}$. For the Petersen graph $P$ we have $r_f(P) = \gamma_2(P) =1$, $r(P)=\omega'(P)=\omega(P)=2$, and $\mu_3(P)=3$.

\begin{proposition} \label{bound_r_f}
If $G$ is a bridgeless cubic graph, then $r_f(G) \leq \gamma_2(G)$.
\end{proposition}

\begin{proof}
Let $M_1$ and $M_2$ be two 1-factors such that $|M_1 \cap M_2| = \gamma_2(G)$, and $F_1$ and $F_2$ be the complementary 2-factors, respectively.
For $i \in \{1,2\}$ let $(D_i,\phi_i)$ be a nowhere-zero $i$-flow on $F_i$. The sum of $(D_1,\phi_1)$ and $(D_2,\phi_2)$
is a 4-flow $(D,\phi)$ on $G$ with $|E(G)-supp(D,\phi)| = \gamma_2(G)$.
\end{proof}

In \cite{s14} Steffen showed that if $G$ is a cyclically 6-edge-connected cubic graph with $\gamma_2(G) \leq 2$, then $G$ has a nowhere-zero 5-flow.
With view on Theorem \ref{omega-mu3} the bound of Proposition \ref{bound_r_f} might not be the best upper bound, see Conjecture \ref{conj:flow_resistance} in Section \ref{Section_Problems}.

Jaeger \cite{Jaeger_88} defined a graph $G$ to be a {\em deletion nowhere-zero 4-flow graph} if it does not have a nowhere-zero 4-flow but
it has an edge $e$, such that $F(G-e) \leq 4$.
He remarked that every deletion nowhere-zero 4-flow graph has a nowhere-zero 5-flow.
Note, that $F(G-e)$ is not always smaller than $F(G)$. For instance $F(K_{3,3})=3$, but $K_{3,3}-e$ is isomorphic to the
complete graph on four vertices, $K_4$, with two
subdivided edges and $F(K_4)=4$; i.e.,~$F(K_{3,3}-e) > F(K_{3,3})$, for all $e \in E(K_{3,3})$.
A snark $G$ is {\it $4$-flow-critical} if it does not admit a nowhere-zero 4-flow
but $G-e$ has a nowhere-zero 4-flow for every $e \in E(G)$. It is easy to see that the flower snarks are $4$-flow-critical.
According to \cite{s98} we say that a snark $G$ is {\it edge-irreducible} if for any two adjacent vertices $x,y \in V(G)$, the graph $G-\{x,y\}$ cannot be extended to a snark by adding edges.

\begin{theorem} \label{charact_4_flow_critical} Let $G$ be a snark. The following three statements are equivalent.
\begin{itemize}
\item[$(i)$]
$G$ is $4$-flow critical.
\item[$(ii)$]
$G$ is cyclically 4-edge-connected and for every $e \in E(G)$ there is a 2-factor $F_e$ of $G$ with precisely two odd circuits which are connected by $e$.
\item[$(iii)$]
$G$ is edge-irreducible.
\end{itemize}
\end{theorem}

\begin{proof}
The equivalence between items 2 and 3 was proved by Steffen in \cite{s98}. It remains to prove the equivalence of the first two statements.
Let $G$ be 4-flow-critical and $e=xy$. Let $G^*$ be obtained from $G-e$ by suppressing the two divalent vertices,
and let $e_x$ and $e_y$ be the two edges where the divalent edges $x$ and $y$ are suppressed. The graph
$G-e$ has a nowhere-zero 4-flow. With the Klein four group as flow values it follows that $G^*$ is 3-edge-colorable.
Hence, there is a 2-factor $F$ which contains $e_x$ and $e_y$. Subdividing $e_x$ and $e_y$ by $x$ and $y$ to adding $e$ to reconstruct $G$
gives a 2-factor $F_e$ of $G$ with two odd circuits which are connected by $e$.
It is easy to see that $G$ is cyclically 4-edge-connected.
\end{proof}

\subsection{Extensions}
Another parameter which measures the complexity of a cubic graph is due to Jaeger \cite{Jaeger_88}.
A graph $G$ is a {\em nearly nowhere-zero 4-flow graph}
if it is possible to add an edge in order to obtain a graph with nowhere-zero 4-flow. Note that a deletion nowhere-zero 4-flow graph is also a nearly nowhere-zero 4-flow graph.
In \cite{s12} Steffen extended this approach to nowhere-zero $r$-flows ($r \in \mathbb{Q}$).
Jaeger's approach can be generalized.
Let $\Phi^+_k(G)$ be the minimum number of edges that
have to be added to a cubic graph $G$ in order to obtain a graph with nowhere-zero k-flow ($k \in \{3,4,5\}$). This parameter is studied by Mohar and \v{S}krekovski in \cite{ms01}.

\begin{theorem} \cite{ms01} \label{ms_extension}
Let $G$ be a loopless cubic graph. If $|V(G)|=n$, then
$\Phi^+_3(G) \leq \lfloor \frac{n}{4} \rfloor$ and $\Phi^+_4(G) \leq \lceil \frac{1}{2} \lfloor \frac{n}{5} \rfloor \rceil$.
\end{theorem}

We will give some upper bounds for $\Phi^+_4$ in terms of oddness and flow resistance.

\begin{theorem}
If $G$ be a bridgeless cubic graph, then $\Phi^+_4(G) \leq \min\{\frac{1}{2}\omega(G), r_f(G)\}$.
\end{theorem}
\begin{proof}
Let $G$ be a cubic graph with $\omega(G)=2t$. Let $F$ be a 2-factor of $G$ with $\omega(G)$ odd circuits.
Let $G/F$ be the multigraph which is obtained from $G$ by contracting the elements of $F$ to vertices. Note that chords in circuits
of $F$ will become loops in $G/F$. Clearly, every odd circuit of $F$ corresponds to a vertex of odd degree in $G/F$. Let $C_1, \dots,C_{2t}$ be the odd circuits of $F$ and
$c_i \in V(C_i)$. Let $G^* = (G+\{c_{2i-1}c_{2i} : i \in \{1, \dots,t\}\})/F$. The multigraph $G^*$ is Eulerian and therefore, it has a nowhere-zero  $\mathbb{Z}_2$-flow. Hence,
$G^*$ has a nowhere-zero $\mathbb{Z}_2 \times \mathbb{Z}_2$-flow, and therefore a nowhere-zero 4-flow. Hence $\Phi^+_4(G) \leq t = \frac{1}{2}\omega(G)$.

Let $(D,\phi)$ a nowhere-zero 4-flow with $|supp(D,\phi)| = |E(G)| - r_f$. Replace every edge $e \not \in supp(D,\phi)$ by a double-edge to obtain a graph $G''$
which has a nowhere-zero 4-flow.
\end{proof}

Mohar and \v{S}krekovski also studied the parameter $\Phi^+_5$. Since every bridgeless cubic graph with oddness 2 has a nowhere-zero 5-flow it follows as above that
$\Phi^+_5(G) \leq \min \{\frac{1}{2}\omega(G) - 1, r_f(G) - 1\}$.

 %%%%%%%%%%%%%%%%%%%%%%%%%%%%% Final remarks

\section{Final remarks and conjectures} \label{Section_Problems}

\subsection{Partial results on the hard conjectures}

Besides the objective to gain new insight into the structure of snarks, complexity measures of bridgeless cubic class 2 graphs also allow us to deduce partial results with respect to the aforementioned conjectures. In the following we list the current status of the results with respect to the conjectures formulated in the Introduction.

\begin{theorem}
If $G$ is a possible minimum counterexample to Conjecture \ref{5flow_conj} (5-flow conjecture), then
\begin{itemize}
%\item
%$G$ is a cubic graph \cite{Seymour_81}.
\item[$(i)$]
$G$ is cyclically 6-edge connected \cite{Kochol_04}.
\item[$(ii)$]
The cyclic connectivity of $G$ is at most $\frac{5}{2}\omega(G)-4$ \cite{s10}.
\item[$(iii)$]
$G$ has girth at least 11 \cite{Kochol_10}.
\item[$(iv)$]
$G$ has oddness  $\omega(G)\ge 6$. \cite{mast14}
\end{itemize}
\end{theorem}

\begin{theorem}
If $G$ is a possible minimum counterexample to Conjecture \ref{Berge_Conjecture} (Berge conjecture), then
\begin{itemize}
\item[$(i)$]
$\omega(G) \geq 2$.
\item[$(ii)$]
if $G$ does not have a non-trivial 3-edge-cut, then $\mu_3(G) \geq 5$ \cite{s15}.
\end{itemize}
\end{theorem}

As far as we know, there are no partial results for Conjecture \ref{Fulkerson} (Berge-Fulkerson conjecture), besides the trivial ones that a possible minimum counterexample is cyclically 4-edge connected and it has girth at least 5.

\begin{theorem}
If $G$ is a possible minimum counterexample to Conjecture \ref{Fan_Raspaud} (Fan-Raspaud conjecture), then
\begin{itemize}
\item[$(i)$]
$\omega(G) \geq 4$ \cite{ms14}.
\item[$(ii)$]
$\mu_3(G) \geq 7$ \cite{s15}.
\item[$(iii)$]
$\chi_e(G) \geq 5$.
\end{itemize}
\end{theorem}

\begin{theorem}
If $G$ is a possible minimum counterexample to Conjecture \ref{5-cdcc} (5-cycle-double-cover conjecture), then
\begin{itemize}
\item[$(i)$]
$\omega'(G) \geq 6$ \cite{huck01}.
\item[$(ii)$]
$\chi_e(G) \geq 5$ \cite{s15}.
\end{itemize}
\end{theorem}

\subsection{Conjectures and problems for bridgeless cubic class 2 graphs}

\subsubsection*{Petersen graph}

We start with some conjectures and problems which are related to the Petersen graph.

\begin{problem} [\cite{Hagglund_2012}]
\label{p_exc_index}
Is the Petersen graph the only cyclically 5-edge-connected snark with excessive index 5?
\end{problem}

The Petersen graph has circular flow number 5, see e.g.~\cite{s01_1}. All other
known snarks with circular flow number 5 have cyclic connectivity 4 (see Esperet, Mazzuoccolo, and Tarsi \cite{emt16}, and M\'a\v{c}ajov\'a and  Raspaud \cite{mr06}).

\begin{problem}
\label{p_circ_flow}
Is the Petersen graph the only cyclically 5-edge-connected snark with circular flow number 5?
\end{problem}

A bridgeless cubic class 2 graph $G$ is {\it vertex-irreducible}, if for any two vertices $x,y \in V(G)$, the graph $G-\{x,y\}$ cannot be extended to a bridgeless cubic class 2 graph by adding edges.
Notice that, according to the definitions in Section \ref{sec:defi}, if $G$ has two disjoint conflicting zones, then it cannot be neither edge- nor vertex-irreducible.

\begin{conjecture} [\cite{s98}] \label{conj:P_vertex_red}
The Petersen graph is the only vertex-irreducible bridgeless cubic class 2 graph.
\end{conjecture}

In \cite{s98} it is proved that a vertex-irreducible cubic graph is cyclically 5-edge connected.
Conjecture \ref{conj:P_vertex_red} is true for all cubic bridgeless graphs with at most 36 vertices \cite{Jan_2016}.

\subsubsection*{Bridgeless cubic class 2 graphs}

Conjectures on general properties of bridgeless cubic class 2 graphs are the following.

\begin{conjecture} [\cite{jsw80}]
If $G$ is a snark, then its cyclic connectivity is at most 6.
\end{conjecture}

The following problem relates problems \ref{p_exc_index} and \ref{p_circ_flow} to each other

\begin{problem}[\cite{aklm16}]
Let $G$ be a bridgeless cubic class 2 graph. Is it true that if $\chi_e(G) = 5$, then $G$ has circular flow number 5?
\end{problem}

We propose the following conjecture:

\begin{conjecture}
There is $\epsilon > 0$, such that $\chi_e(G) \leq 4$ for every cubic graph $G$ with circular flow number smaller than $4 + \epsilon$.
\end{conjecture}

\subsubsection*{Measures of edge-uncolorability}

We start with three specific problems of Lukot'ka and Maz\'{a}k \cite{LM_2016}.

\begin{problem} [\cite{LM_2016}]
\label{conjec-oddness-res}
Does there exist a cubic graph with weak oddness 4 and oddness at least 6?
\end{problem}

\begin{problem} [\cite{LM_2016}]
For which integers $k \geq 3$ does there exist a cyclically $k$-edge-connected cubic class 2 graph $G$ with $\omega(G) \not= \omega'(G)$.
\end{problem}

\begin{problem} [\cite{LM_2016}]
In a 3-edge-connected cubic class 2 graph, can the expansion of a vertex into a triangle decrease the oddness?
\end{problem}

A graph $G$ is {\em hypohamiltonian} if it is not hamiltonian but $G-v$ is hamiltonian for every vertex $v \in V(G)$.
For instance, the Petersen graph and the flower snarks are hypohamiltonian. M\'a\v{c}ajov\'a and \v{S}koviera \cite{M_Skoviera_2007} showed that
cubic hypohamiltonian class 2 graphs are cyclically 4-edge-connected and have girth at least 5 (i.e.~they are snarks),
and that there are cyclically 6-edge-connected hypohamiltonian snarks with girth 6.
If $G$ is a hypohamiltonian snark, then $r(G) = \omega(G) = 2$, and $G$ satisfies Conjecture \ref{Fan_Raspaud}.
If the following conjecture is true, then hypohamiltonian snarks have excessive index 5, i.e.~they satisfy Conjecture \ref{Berge_Conjecture}.

\begin{conjecture} [\cite{s15}]
\label{hypo}
Let $G$ be a cubic class 2 graph. If $G$ is hypohamiltonian, then $\mu_3(G) = 3$.
\end{conjecture}

Conjecture \ref{hypo} is verified for all hypohamiltonian class 2 graphs with at most 36 vertices by Goedgebeur and Zamfirescu \cite{Jan_Carol_2016}.

\begin{conjecture}
\label{conj:flow_resistance}
If $G$ is a bridgeless cubic graph, then $r_f(G) \leq r(G)$.
\end{conjecture}

Let ${\cal{S}}$ be the set of snarks, and $\tau$ be one of the complexity measures which are discussed in this paper. Let
$$
q(\tau,k) = \max \left\{\frac{k}{|V(G)|} : G \in {\cal{S}} \text{ and } \tau(G) \geq k \right\}.
$$

\begin{problem}
Let $k \geq 1$ be an integer. Determine $q(\tau,k)$.
\end{problem}

\begin{problem}
What is the larges value $c$ such that $q(\tau,k) \geq c$ for all $k > 0$?
\end{problem}

These two questions were asked by H\"agglund \cite{Hagglund_2012} for the oddness.
Clearly, for the oddness it suffices to consider even numbers.
The Petersen graph is the smallest snark.
Hence, $q(r_f,1) = q(\gamma_2,1) = \frac{1}{10}$ and $q(r,2) = q(\omega,2) = q(\omega',2) = \frac{1}{5}$,
and $q(\mu_3,3) = \frac{3}{10}$.

The smallest snark with oddness 4 has at least 38 vertices (see Brinkmann, Goedgebeur, H\"agglund, and Markstr\"om \cite{Gunnar_2011_paper}) and at most 44 vertices (see Lukot'ka, M\'{a}\v{c}ajov\'{a}, Maz\'{a}k, and \v{S}koviera \cite{lmms_13}).
Hence, $\frac{2}{19} \leq q(\omega,4) \leq \frac{1}{11}$. H\"agglund \cite{Hagglund_2012} proved
that $q(\omega,k) \geq \frac{1}{18}$ and for multiples of 6 he improved this result to $q(\omega,6k) \geq \frac{1}{15}$.

In \cite{lmms_13} the reciprocal parameter
oddness ratio $|V(G)|/\omega(G)$ and the resistance ratio $|V(G)|/r(G)$ are also studied.
They adopt the asymptotical approach of Steffen \cite{s04} for a sophisticated analysis of these parameters in the case of cyclically $k$-edge connected snarks,
$k \in \{2,\dots,6\}$. More precisely, they study the parameters
$$
A_{\omega}=\lim_{|V(G)| \rightarrow \infty} \inf \frac{|V(G)|}{\omega(G)}\qquad \mbox{and}\qquad
A_{r} = \lim_{|V(G)| \rightarrow \infty} \inf \frac{|V(G)|}{r(G)}.
$$
%xxxx
%include table of \cite{lmms_13}
%xxxx
\begin{table}
\begin{center}
\begin{tabular}{|c||c|c|}
\hline
connectivity & lower bound & current upper bound\\
\hline \hline
2 &  5.41 & 7.5\\
\hline
3 & 5.52 & 9\\
\hline
4 & 5.52 & 13\\
\hline
5 & 5.83 & 25\\
\hline
6 & 7 & 99\\
\hline
\end{tabular}
\caption{Lower and upper bounds for the oddness ratio \cite{lmms_13}}
\end{center}
\end{table}

\begin{conjecture} [\cite{lmms_13}]
\label{conj:lmms_13_44}
Let $G$ be a snark. If $\omega(G) = 4$, then $|V(G)| \geq 44$.
\end{conjecture}

If Conjecture \ref{conj:lmms_13_44} is true, then it is best possible. Lukot'ka, M\'{a}\v{c}ajov\'{a}, Maz\'{a}k, and \v{S}koviera
\cite{lmms_13} constructed a snark with 44 vertices and oddness 4.

\begin{conjecture} [\cite{lmms_13}]
If $G$ is a snark, then $|V(G)| \geq \left(7+\frac{1}{2}\right)\omega(G)$.
\end{conjecture}

\begin{problem}
Let $\tau_1$ and $\tau_2$ be two complexity measures of snarks such that $\tau_1(G) \leq \tau_2(G)$ for all snarks $G$.
Does there exist a function $f$ such that $\tau_2(G) \leq f(\tau_1(G))$.
\end{problem}

That question was studied in \cite{s04}, where it is asked whether there is a function $f$ such that $\omega(G) \leq f(r(G))$
for each bridgeless cubic graph. Furthermore there is shown that there is no constant $c$ such that $1 \leq c < 2$ and $\omega(G) \leq cr(G)$,
for every bridgeless cubic class 2 graph $G$. We conjecture the following two statements to be true.

\begin{conjecture}
If $G$ is a bridgeless cubic graph, then $\omega'(G) \leq 2r(G)$.
\end{conjecture}

\begin{conjecture}
If $G$ is a snark, then $\omega(G) \leq 2r(G)$.
\end{conjecture}

It might be interest to extend the definition of Mohar and \v{S}krekovski in \cite{ms01}.
Let $\Phi^+_r(G)$ be the minimum number of edges that
have to be added to a cubic graph $G$ in order to obtain a graph with nowhere-zero $r$-flow ($r < 6$).

\begin{problem}
Does there exist $s$ with $4 < s < 5$ such that $\Phi^+_r(G) = \Phi^+_4(G)$ for all $r$ with $4 < r < s$?
\end{problem}

\subsection*{Acknowledgements} Research supported by the
{\em Ministerio de Ciencia e Innovaci\'on}, Spain, and the
{\em European Regional Development Fund} under project MTM2014-60127-P,
and the {\em Catalan Research Council} under project 2014SGR1147.
The research of the third author on this project was supported by Deutsche Forschungsgemeinschaft (DFG) grant STE 792/2-1.
%The authors wish to thank ....

\end{document}